\documentclass[a4paper,10pt]{amsart}
\usepackage[utf8]{inputenc}

\usepackage[
  margin=30mm,
  marginparwidth=25mm,     
  marginparsep=2mm,       
  bottom=25mm,
  ]{geometry}

\usepackage[bbgreekl]{mathbbol}
\usepackage{amsfonts}
\usepackage{latexsym,amssymb,amsthm,mathrsfs,amsmath,amscd,enumerate,enumitem, amscd,color}
\usepackage{mathrsfs}
\usepackage{tikz}
\usepackage{hyperref}
\usetikzlibrary{plotmarks}
\usepackage{soul}

\DeclareSymbolFontAlphabet{\mathbb}{AMSb}
\DeclareSymbolFontAlphabet{\mathbbl}{bbold}
\DeclareSymbolFontAlphabet{\mathbb}{AMSb}
\DeclareSymbolFontAlphabet{\mathbbl}{bbold}
\DeclareFontEncoding{FMS}{}{}
\DeclareFontSubstitution{FMS}{futm}{m}{n}
\DeclareFontEncoding{FMX}{}{}
\DeclareFontSubstitution{FMX}{futm}{m}{n}
\DeclareSymbolFont{fouriersymbols}{FMS}{futm}{m}{n}
\DeclareSymbolFont{fourierlargesymbols}{FMX}{futm}{m}{n}
\DeclareMathDelimiter{\VERT}{\mathord}{fouriersymbols}{152}{fourierlargesymbols}{147}

\newtheorem{ThA}{Proposition}

\newtheorem{thm}{Theorem}[section]

 \newtheorem{prop}[thm]{Proposition}
\theoremstyle{definition}

 \theoremstyle{remark}
 \newtheorem{rem}[thm]{Remark}

\usepackage{hyperref}

\newcommand{\supp}{\mathop{\mathrm{supp}}}

\newcommand{\essinf}{\mathop{\mathrm{ess\, inf \;}}}

\numberwithin{equation}{section}
\allowdisplaybreaks

\begin{document}


\title[]
 {Variation operators associated with the semigroups generated by Schr\"odinger operators with inverse square potentials}

\author[V. Almeida]{V\'{\i}ctor Almeida}

\author[J.J. Betancor]{Jorge J. Betancor}
\address{V\'{\i}ctor Almeida, Jorge J. Betancor, Lourdes Rodr\'{\i}guez-Mesa\newline
	Departamento de An\'alisis Matem\'atico, Universidad de La Laguna,\newline
	Campus de Anchieta, Avda. Astrof\'isico S\'anchez, s/n,\newline
	38721 La Laguna (Sta. Cruz de Tenerife), Spain}
\email{valmeida@ull.edu.es, jbetanco@ull.es, lrguez@ull.edu.es
}

\author[L. Rodr\'{\i}guez-Mesa]{Lourdes Rodr\'{\i}guez-Mesa}

\thanks{The authors are partially supported by grant PID2019-106093GB-I00 from the Spanish Government}



\date{}


\begin{abstract}
By $\{T_t^a\}_{t>0}$ we denote the semigroup of operators generated by the Friedrichs extension of the Schr\"odinger operator with the inverse square potential $L_a=-\Delta +\frac{a}{|x|^2}$ defined in $C_c^\infty (\mathbb{R}^n\setminus\{0\})$. In this paper we establish weighted $L^p$-inequalities for the maximal, variation, oscillation and jump operators associated with $\{t^\alpha\partial_t^\alpha T_t^a\}_{t>0}$, where $\alpha \geq 0$ and $\partial _t^\alpha$ denotes the Weyl fractional derivative. The range of values $p$ that works is different when $a\geq 0$ and when $-\frac{(n-2)^2}{4}< a<0$.
\end{abstract}

\maketitle

\baselineskip=15pt
\section{Introduction}
We consider the Schr\"odinger operator $L_a$ with inverse square potential defined by
$$
L_a=-\Delta +\frac{a}{|x|^2},
$$
on $C_c^\infty (\mathbb{R}^n\setminus\{0\})$, the space of smooth functions with compact support in $\mathbb{R}^n\setminus\{0\}$. Here $\Delta$ denotes as usual the Euclidean Laplacian operator on $\mathbb{R}^n$ and $a> -\frac{(n-2)^2}{4}$.
In \cite[Theorem 2]{Si} it was proved that the realization $\mathcal{L}_{a,2}$ of $L_a$ in $L^2(\mathbb{R}^n)$ is essentially selfadjoint on $C_c^\infty (\mathbb{R}^n\setminus\{0\})$ if and only if $a\geq  1-\frac{(n-2)^2}{4}$. Perturbation techniques can be used to see that $\mathcal{L}_{a,2}$ is selfadjoint when $a>1-\frac{(n-2)^2}{4}$ (\cite[Theorem 6.8]{OK1}). For every $1<p<\infty$ we denote by $\mathcal{L}_{a,p}$ the realization of $L_a$ in $L^p(\mathbb{R}^n)$ with domain the Sobolev space $W^{2,p}(\mathbb{R}^n)$. In \cite[Theorem 3.11]{OK2} it was proved that $\mathcal{L}_{a,p}$ generates a contractive and positive $C_0$-semigroup in $L^p(\mathbb{R}^n)$ being $C_c^\infty(\mathbb{R}^n)$ a core for $\mathcal{L}_{a,p}$ provided that $n>2p$ and $a>-\frac{(p-1)(n-2p)n}{p^2}$. Other generation results for suitable realizations of $L_a$ in $L^p(\mathbb{R}^n)$ can be found in \cite{FGR}.

In \cite{LSV} and \cite{SV} (see also \cite{V}) positive $C_0$-semigroups associated with second order uniformly elliptic divergence type operators with singular lower order terms, including the operator $L_a$, subject to a wide class of boundary conditions are studied. They described a method of constructing positive $C_0$-semigroups on $L^p$ associated to sesquilinear forms that are not necessarily sectorial. The boundary condition appears in the domain of the form. Assume that $-\frac{(n-2)^2}{4}<a<0$. We denote by $p_+$ and $p_-$ the two roots of the equation 
$$
x^2-x-\frac{a}{(n-2)^2}=0,
$$
being $p_-<p_+$. Then, for every $p\in [(\frac{n}{n-2}p_-')',\frac{n}{n-2}p_+]$, the realization $\mathbb{L}_a$ of $L_a$ generates a quasicontractive $C_0$-semigroup associated with the form $\tau$ defined by
$$
\tau(u,v)=\langle \nabla u,\nabla v\rangle+a\langle |x|^{-2}u,v\rangle,\quad u,v\in W^{1,2}(\mathbb{R}^n)\cap \big\{f\in L^2(\mathbb{R}^n):|x|^{-1}f\in L^2(\mathbb{R}^n)\big\}. 
$$
Furthermore, the $p$-interval is optimal (\cite[p. 88]{V}). This situation remembers the named pencil phenomenon described in \cite{MST} (see also \cite{NSj}).

As in \cite{KMVZZ} we denote by $\mathcal{L}_a$ the Friedrichs extension of $L_a$. Thus, $\mathcal{L}_a$ is the unique selfadjoint extension of $L_a$ whose form domain $Q(\mathcal{L}_a)=D(\sqrt{\mathcal{L}_a})\subset L^2(\mathbb{R}^n)$ coincides with the completion of $C_c^\infty(\mathbb{R}^n\setminus\{0\})$ with respect to the norm
$$
\|f\|_{Q(\mathcal{L}_a)}=\left(\int_{\mathbb{R}^n}\Big(|\nabla f(x)|^2+\Big(1+\frac{a}{|x|^2}\Big)|f(x)|^2\Big)dx\right)^{1/2},\quad f\in C_c^\infty (\mathbb{R}^n\setminus\{0\}).
$$
Furthermore, $\mathcal{L}_a$ is positive. If $a\geq 1-\frac{(n-2)^2}{4}$, $L_a$ admits a one parameter family of selfadjoint extensions and $\mathcal{L}_a$ is characterized by the asymptotics in the origen of spherically symmetric eigenfunctions. For further details, see \cite[\S X.3]{RS}. 

We denote by $\{T_t^a\}_{t>0}$ the semigroup of operators generated by $-\mathcal{L}_a$ in $L^2(\mathbb{R}^n)$ (\cite{Pil}). For every $t>0$, $T_t^a$ admits an integral representation given by
$$
T_t^a(f)(x)=\int_{\mathbb{R}^n}T_t^a(x,y)f(y)dy,\quad f\in L^2(\mathbb{R}^n).
$$
According to the results in \cite{LS} and \cite{MS} (see \cite[Theorem 2.1]{KMVZZ}) there exist $C,c_1,c_2>0$ such that, for every $t>0$ and $x,y\in \mathbb{R}^n\setminus\{0\}$,
\begin{equation}\label{1.1}
\frac{1}{C}\Big(1+\frac{\sqrt{t}}{|x|}\Big)^\sigma\Big(1+\frac{\sqrt{t}}{|y|}\Big)^\sigma\frac{e^{-c_2\frac{|x-y|^2}{t}}}{t^{n/2}}\leq T_t^a(x,y)\leq C\Big(1+\frac{\sqrt{t}}{|x|}\Big)^\sigma\Big(1+\frac{\sqrt{t}}{|y|}\Big)^\sigma\frac{e^{-c_2\frac{|x-y|^2}{t}}}{t^{n/2}},
\end{equation}
where $\sigma=\frac{n-2}{2}-\frac{1}{2}\sqrt{(n-2)^2+4a}$. Note that if $\sigma\leq 0$, that is, if $a\geq 0$, then \eqref{1.1} implies Gaussian upper bounds for $\{T_t^ a\}_{t>0}$. In this case, for every $t>0$, $T_t^a$ is bounded from $L^p(\mathbb{R}^n)$ into itself, for each $1\leq p\leq \infty$. However, if $-\frac{(n-2)^2}{4}< a<0$, then $0<\sigma <\frac{n-2}{2}$ and neither Gaussian upper bounds nor Poisson upper bounds are satisfied for $\{T_t^a\}_{t>0}$. By \eqref{1.1} it can be deduced that $T_t^a$, $t>0$, is not bounded from $L^p(\mathbb{R}^n)$ into itself for $p\geq n/\sigma$, when $0<\sigma <\frac{n-2}{2}$. Indeed, assume that $0<\sigma<\frac{n-2}{2}$. By \eqref{1.1} we deduce that 
\begin{align*}
    \int_\mathbb{\mathbb{R}}\mathcal{X}_{B(0,1)}(y)T_t^a(x,y)dy&\geq C\Big(1+\frac{\sqrt{t}}{|x|}\Big)^\sigma \int_{B(0,1)}\Big(1+\frac{\sqrt{t}}{|y|}\Big)^\sigma \frac{e^{-c\frac{|x-y|^2}{t}}}{t^{n/2}}dy\\
    &\geq C\Big(\frac{\sqrt{t}}{|x|}\Big)^\sigma\int_{B(0,1)}\Big(\frac{\sqrt{t}}{|y|}\Big)^\sigma \frac{e^{-\frac{c}{t}}}{t^{n/2}}dy\\
    &\geq \frac{C(t)}{|x|^\sigma},\quad x\in B(0,1)\setminus\{0\}\mbox{ and }t>0.
\end{align*}
Then, for every $t>0$, $T_t^a(\mathcal{X}_{B(0,1)})\not \in L^p(\mathbb{R}^n)$ when $p\geq n/\sigma$.

The semigroup of operators $\{T_t^a\}_{t>0}$ can be extended as an analytic semigroup $\{T_z\}_{z\in \Sigma _{\pi/4}}$, where $\Sigma_{\pi/4}=\{z\in \mathbb{C}, |\rm{arg }\;z|<\frac{\pi}{4}\}$. According to \cite[Proposition 3.4]{BADLL} we have that there exist $C,c>0$ for which
\begin{equation}\label{1.2}
    |T_z^a(x,y)|\leq C\left(1+\frac{\sqrt{|z|}}{|x|}\right)^\sigma\left(1+\frac{\sqrt{|z|}}{|y|}\right)^\sigma |z|^{-n/2}e^{-c\frac{|x-y|^2}{|z|}},\quad z\in \Sigma_{\frac{\pi}{4}},\;x,y\in \mathbb{R}^n\setminus\{0\}.
\end{equation}
By using for instance Cauchy integral formula, from \eqref{1.2} we can deduce that, for certain $C,c>0$,
\begin{equation}\label{1.3}
|t^k\partial _t^kT_t^a(x,y)|\leq C\left(1+\frac{\sqrt{t}}{|x|}\right)^\sigma\left(1+\frac{\sqrt{t}}{|y|}\right)^\sigma t^{-n/2}e^{-c\frac{|x-y|^2}{t}},\quad x,y\in \mathbb{R}^n\setminus\{0\},\;t>0.
\end{equation}
We remark that as far as we know we have not H\"older continuity properties with respect to the spatial variables for $T_t(x,y)$. Then, the Calder\'on-Zygmund theory for singular integrals can not be used to prove $L^p$-boundedness properties of harmonic analysis operator (maximal operators and Littlewood-Paley-Stein functions associated with $\{T_t^a\}_{t>0}$, Riesz transforms, spectral multipliers,...) in the $\mathcal{L}_a$-setting. Fractional powers of $\mathcal{L}_a$, some square functions, Riesz transforms and smoothing inequalities in $L^p(\mathbb{R}^n)$ were studied in \cite{BADLL}. Mihlin multipliers associated with $\mathcal{L}_a$ were considered in \cite{KMVZZ} where  $\mathcal{L}_a$-Sobolev spaces were also investigated. In \cite{MZZ} maximal $L^p$ estimates for the solution to initial value problems of the linear Schr\"odinger with inverse square potential were studied. Recently, Bui (\cite{Bui1}) has introduced Besov and Triebel-Lizorkin spaces for the operator $\mathcal{L}_a$. 

In \cite{BBD} H.-Q. Bui, T.-A. Bui and Duong studied Besov and Triebel Lizorkin associated with operators on homogeneous spaces. For every $a\geq 0$, the operator $\mathcal{L}_a$ can be seen as a special case of the operators considered in \cite{BBD}. In this paper we are going to consider the Hardy space $H^1(\mathcal{L}_a)$, for every $a\geq 0$. Let $a\geq 0$. As in \cite[\S 5.2]{BBD} we define $H^1(\mathcal{L}_a)$ as the completion of the set $\{f\in L^2(\mathbb{R}^n): S_a(f)\in L^1(\mathbb{R}^n)\}$ with respect to the norm
$$
\|f\|_{H^1(\mathcal{L}_a)}=\|S_a(f)\|_{L^1(\mathbb{R}^n)},
$$
where
$$
S_a(f)(x)=\left(\int_0^\infty \int_{|x-y|<t}|t^2\mathcal{L}_ae^{-t\mathcal{L}_a^2}(f)(y)|^2\frac{dydt}{t^{n+1}}\right)^{1/2},\quad x\in \mathbb{R}^n.
$$
The space $H^1(\mathcal{L}_a)$ is also a special case of the one considered in \cite{HLMMY}. $H^1(\mathcal{L}_a)$ can be characterized by using atoms. Let $M\in \mathbb{N}$. A function $b\in L^2(\mathbb{R}^n)$ is said to be a $(2,M)$-atom associated to the operator $\mathcal{L}_a$ when there exists a function $u$ belonging to the domain $D(\mathcal{L}_a^M)$ of $\mathcal{L}_a^M$ and a ball $B$ such that

(i) $b=\mathcal{L}_a^Mu$;

(ii) $\supp \mathcal{L}_a^ku \subset B$, $k=0,1,...,M$;

(iii) $\|(r_B^2\mathcal{L}_a)^ku\|_{L^2(\mathbb{R}^n)}\leq r_B^{2M}|B|^{-1/2}$, $k=0,1,...,M$, where $r_B$ denotes the radius of $B$.

The atomic Hardy spaces $H_{at, M}^1(\mathcal{L}_a)$ is defined in the usual way. A function $f\in L^2(\mathbb{R}^n)$ is in $\mathbb{H}_{at,M}^1(\mathbb{R}^n)$ when $f=\sum_{j=0}^\infty \lambda_ja_j$ in $L^2(\mathbb{R}^n)$ where, for every $j\in \mathbb{N}$, $a_j$ is a $(2,M)$-atom for $\mathcal{L}_a$ and $\lambda _j>0$ being $\sum_{j=0}^\infty \lambda _j<\infty$. The norm $\|\cdot\|_{H_{at,M}^1(\mathcal{L}_a)}$ is defined on $\mathbb{H}_{at,M}^1(\mathbb{R}^n)$ by
$$
\|f\|_{H_{at,M}^1(\mathcal{L}_a)}=\inf \sum_{j=0}^\infty |\lambda _j|,\quad f\in H_{at,M}^1(\mathcal{L}_a),
$$
where the infimum is taken over all the sequences $\{\lambda_j\}_{j\in \mathbb{N}}\subset (0,\infty )$ such that $\sum_{j=1}^\infty \lambda_j<\infty$ and $f=\sum_{j=0}^\infty \lambda_ja_j$ in $L^2(\mathbb{R}^n)$ where $a_j$ is a $(2,M)$-atom associated with $\mathcal{L}_a$, for every $j\in \mathbb{N}$. The Hardy space $H_{at,M}^1(\mathcal{L}_a)$ is defined as the completion of $\mathbb{H}_{at,M}^1(\mathcal{L}_a)$ with respect to $\|\cdot\|_{H_{at,M}^1(\mathcal{L}_a)}$.

According to \cite[Theorem 2.5]{HLMMY}, for every $M>1$, $H^1(\mathcal{L}_a)=H_{at,M}^1(\mathcal{L}_a)$ algebraic and topologically. The same result holds when the condition (iii) is replaced by

(iii') $\|(r_B^2\mathcal{L}_a)^k u\|_{L^q(\mathbb{R}^n)}\leq r_B^{2M}|B|^{1/q-1}$, $k=0,1,...,M$, \newline
where $1<q\leq \infty$ (\cite[Theorem 1.4]{SY1}).

We now consider the maximal operator $T_*^a$ defined by
$$
T_*^a(f)=\sup _{t>0}|T_t^a(f)|,\quad f\in L^p(\mathbb{R}^n),\;1\leq p\leq \infty.
$$
According to \eqref{1.1}, when $a\geq 0$ $T_*^a$ is bounded from $L^p(\mathbb{R}^n)$ into itself for every $1<p\leq \infty$ and from $L^1(\mathbb{R}^n)$ into $L^{1,\infty }(\mathbb{R}^n)$. We define $\|\cdot \|_{H^1_h(\mathcal{L}_a)}$ by
$$
\|f\|_{H^1_h(\mathcal{L}_a)}=\|T_*^a(f)\|_{L^1(\mathbb{R}^n)},\quad f\in L^2(\mathbb{R}^n).
$$

The radial heat maximal $H_h^1(\mathcal{L}_a)$ is defined as the completion of the set $\{f\in L^2(\mathbb{R}^n):\|f\|_{H^1_h(\mathcal{L}_a)}<\infty\}$ with respect to $\|\cdot \|_{H^1_h(\mathcal{L}_a)}$. By \cite[Theorem 1.3]{SY2} we have that $H^1_h(\mathcal{L}_a)=H^1(\mathcal{L}_a)$, algebraic and topologically. 

Suppose now that $\{S_t\}_{t>0}$ is a family of bounded operators in $L^p(\mathbb{R}^n)$ for some $1\leq p<\infty$. If $\{t_j\}_{j\in \mathbb{N}}\subset (0,\infty )$ is a decreasing sequence, the oscillation operator $O(\{S_t\}_{t>0},\{t_j\}_{j\in \mathbb{N}})$ is defined by
$$
O(\{S_t\}_{t>0},\{t_j\}_{j\in \mathbb{N}})(f)(x)=\left(\sum_{j\in \mathbb{N}}\sup_{t_{j+1}\leq \varepsilon _{j+1}<\varepsilon _j\leq t_j}\big|S_{\varepsilon _j}(f)(x)-S_{\varepsilon _{j+1}}(f)(x)\big|^2\right)^{1/2}, \quad x\in \mathbb{R}^n.
$$
Let $\rho >2$. The variation operator $V_\rho(\{S_t\}_{t>0})$ is defined by 
$$
V_\rho(\{S_t\}_{t>0})(f)(x)=\sup_{\begin{array}{c}  0<t_k<...<t_1<\infty \\
k\in\mathbb N \end{array}}\left(\sum_{j=1}^{k-1}\big|S_{t_j}(f)(x)-S_{t_{j+1}}(f)(x)\big|^\rho\right)^{1/\rho}, \quad x\in \mathbb{R}^n.
$$
A remarkable question is concerned with the measurability of the functions we just defined as it is commented in \cite[p. 60]{CJRW1}.

For every $\lambda >0$, the $\lambda$-jump operator $\Lambda(\{S_t\}_{t>0},\lambda)$ is given through
\begin{align*}
\Lambda(\{S_t\}_{t>0},\lambda )(f)(x)=\sup\Big\{&n\geq 0: \mbox{ there exist }s_1<t_1\leq s_2<t_2\leq ...\leq s_n<t_n\mbox{ such that }\\
&|S_{t_i}(f)(x)-S_{s_i}(f)(x)|>\lambda ,\;i=1,2,...,n\Big\},\quad x\in \mathbb{R}^n.
\end{align*}
The size of $\Lambda(\{S_t\}_{t>0},\lambda )(f)(x)$ give us information on the way of the family $\{S_t\}_{t>0}$ converges.

For every $k\in \mathbb{Z}$ we consider
$$
V_k(\{S_t\}_{t>0})(f)(x)=\sup_{2^{-k}<t_\ell<...<t_1<2^{-k+1}}\left(\sum_{j=1}^{\ell-1}\big|S_{t_j}(f)(x)-S_{t_{j+1}}(f)(x)\big|^2\right)^{1/2 }, \quad x\in \mathbb{R}^n.
$$
The short variation $S_V(\{S_t\}_{t>0})$ is defined by
$$
S_V(\{S_t\}_{t>0})(f)(x)=\left(\sum_{k=-\infty}^{+\infty}\big(V_k(\{S_t\}_{t>0})(f)(x)\big)^2\right)^{1/2},\quad x\in \mathbb{R}^n.
$$
The oscillation and short variation can be defined by replacing the exponent by other exponent $\rho\not=2$. If $\rho >2$ the new operators are controlled by those ones with $\rho =2$ and the estimations for the operators with $\rho=2$ are transferred to the operators with $\rho >2$. When $\rho <2$ it is usual that the new operators are not bounded in $L^p(\mathbb{R}^n)$ (see \cite{AJS} for the case of differentiation operators). We also need, in general, to consider $\rho >2$ in the definition of the variation operator in order to get $L^p$-boundedness properties (see, for instance, \cite{Qi} for the case of martingales).

Variational inequalities appears in many recent papers in probability, ergodic theory and harmonic analysis. Bourgain \cite{Bo} was the first to consider variation inequalities in ergodic theory. Bourgain's inequalities had as predecessor to Lepingl\'e's inequality \cite{Lep} for martingales and they were extended to $L^p$ with $p\not=2$ in \cite{JKRW}. Variation and oscillation operators for families of operators associated with singular integrals and semigroups of operators have been studied. Each of them informs on some convergence for the family of operators under consideration. The interested reader can see in particular \cite{CJRW1}, \cite{CJRW2}, \cite{CMMTV}, \cite{JR}, \cite{JSW}, \cite{JW}, \cite{LeMX}, \cite{MTX}, \cite{OSTTW} and references therein.

Let $\alpha \geq 0$. We choose $m\in \mathbb{N}$ such that $m-1\leq \alpha <m$. If $\phi \in C^m(0,\infty)$ the $\alpha$-th Weyl derivative $\partial_t ^\alpha \phi$ of $\phi$ is defined by
$$
\partial_t^\alpha \phi (t)=\frac{e^{-i\pi(m-\alpha)}}{\Gamma (m-\alpha)}\int_0^\infty \partial _u^m\phi (u)_{|u=t+s}s^{m-\alpha -1}ds,\quad t\in (0,\infty ).
$$
We recall the definitions of the classes of weights that we consider. Suppose that $1<q<\infty$. A nonnegative locally integrable function $w$ defined on $\mathbb{R}^n$ is said to be in the Muckenhoupt class $A_q(\mathbb{R}^n)$ when there exists $C>0$ such that, for every ball $B$ in $\mathbb{R}^n$, 
$$
\frac{1}{|B|}\int_Bw(x)dx\left(\frac{1}{|B|}\int_B(w(x))^{-\frac{1}{q-1}}dx\right)^{q-1}\leq C.
$$
A nonnegative locally integrable function $w$ defined in $\mathbb{R}^n$ is in the Muckenhoupt class $A_1(\mathbb{R}^n)$ when there exists $C>0$ such that, for every ball $B$ in $\mathbb{R}^n$,
$$
\frac{1}{|B|}\int_Bw(x)dx\leq C\essinf_{x\in B}w(x),
$$
and it is in the Muckenhoupt class $A_\infty (\mathbb{R}^n)$ when $w\in A_q(\mathbb{R}^n)$ for some $1\leq q<\infty$. The main properties of the weights in the Muckenhoupt classes can be found in \cite{Duo}.

A nonnegative locally integrable function $w$ defined in $\mathbb{R}^n$ is said to be in the reverse H\"older class $RH_q(\mathbb{R}^n)$, where $1<q<\infty$, if there exists $C>0$ such that, for every ball $B$ in $\mathbb{R}^n$,
$$
\left(\frac{1}{|B|}\int_Bw(x)^qdx\right)^{1/q}\leq \frac{C}{|B|}\int_Bw(x)dx.
$$
We say that $w$ is in the reverse H\"older class $RH_\infty(\mathbb{R}^n)$ when there exists $C>0$ such that, for every ball $B$ in $\mathbb{R}^n$,
$$
w(x)\leq \frac{C}{|B|}\int_Bw(y)dy,\quad \mbox{ for almost all }x\in B.
$$
The main properties of the reverse H\"older classes of weights can be found in \cite{JN}.

We now state our results in this paper. They establish the behaviour of the operators \newline $V_\rho (\{t^\alpha \partial _t^\alpha T_t^a\}_{t>0})$, $O(\{t^\alpha \partial _t^\alpha T_t^a\}_{t>0},\{t_j\}_{j\in \mathbb{N}})$, $\lambda \Lambda (\{t^\alpha \partial _t^\alpha T_t^a\}_{t>0},\lambda )^{1/\rho}$, $\lambda >0$ and $S_V(\{t^\alpha \partial _t^\alpha T_t^a\}_{t>0})$ in $L^p(\mathbb{R}^n,w)$, $1\leq p<\infty$, and in $H^1(\mathcal{L}_a)$ when $\alpha \geq 0$ and $a>-\frac{(n-2)^2}{4}$.

\begin{thm}\label{Th1.1}
Let $\alpha \geq 0$ and $\rho >2$. Suppose that $\{t_j\}_{j\in \mathbb{N}}\subset (0,\infty)$ is a decreasing sequence that converges to zero.

(i) Assume that $a\geq 0$. Then, the operators $V_\rho (\{t^\alpha \partial _t^\alpha T_t^a\}_{t>0})$, $O(\{t^\alpha \partial _t^\alpha T_t^a\}_{t>0},\{t_j\}_{j\in \mathbb{N}})$ and $S_V(\{t^\alpha \partial _t^\alpha T_t^a\}_{t>0})$ are bounded from $L^p(\mathbb{R}^n,w)$ into itself, for every $1<p<\infty$ and $w\in A_p(\mathbb{R}^n)\cap RH_1(\mathbb{R}^n)$, and from $L^1(\mathbb{R}^n)$ into $L^{1,\infty}(\mathbb{R}^n)$. The family $\{\lambda \Lambda (\{t^\alpha \partial _t^\alpha T_t^a\}_{t>0},\lambda )^{1/\rho}\}_{\lambda >0}$ is uniformly bounded from $L^p(\mathbb{R}^n,w)$ into itself, for every $1<p<\infty$ and $w\in A_p(\mathbb{R}^n)\cap RH_1(\mathbb{R}^n)$, and from $L^1(\mathbb{R}^n)$ into $L^{1,\infty}(\mathbb{R}^n)$.

(ii) Assume that $-\frac{(n-2)^2}{4}< a<0$. Then, the operators $V_\rho (\{t^\alpha \partial _t^\alpha T_t^a\}_{t>0})$, \newline $O(\{t^\alpha \partial _t^\alpha T_t^a\}_{t>0},\{t_j\}_{j\in \mathbb{N}})$ and $S_V(\{t^\alpha \partial _t^\alpha T_t^a\}_{t>0})$ are bounded from $L^p(\mathbb{R}^n,w)$ into itself, and the family $\{\lambda \Lambda (\{t^\alpha \partial _t^\alpha T_t^a\}_{t>0},\lambda )^{1/\rho}\}_{\lambda >0}$ is uniformly bounded from $L^p(\mathbb{R}^n,w)$ into itself, provided that $p\in (\frac{n}{n-\sigma},\frac{n}{\sigma})$ and $w\in A_{(n-\sigma)p/n}(\mathbb{R}^n)\cap RH_{(\frac{n}{\sigma p})'}(\mathbb{R}^n)$. 
\end{thm}

\begin{thm}\label{Th1.2}
Let $\alpha \geq 0$, $\rho >2$ and $a\geq 0$. Suppose that $\{t_j\}_{j\in \mathbb{N}}\subset (0,\infty)$ is a decreasing sequence that converges to zero. Then, the operators $V_\rho (\{t^\alpha \partial _t^\alpha T_t^a\}_{t>0})$, $O(\{t^\alpha \partial _t^\alpha T_t^a\}_{t>0},\{t_j\}_{j\in \mathbb{N}})$ and $S_V(\{t^\alpha \partial _t^\alpha T_t^a\}_{t>0})$ are bounded from $H^1(\mathcal{L}_a)$ into $L^1(\mathbb{R}^n)$. The family $\{\lambda \Lambda (\{t^\alpha \partial _t^\alpha T_t^a\}_{t>0},\lambda )^{1/\rho}\}_{\lambda >0}$ is uniformly bounded from $H^1(\mathcal{L}_a)$ into $L^1(\mathbb{R}^n)$.

If $f\in L^1(\mathbb{R}^n)$, then $f\in H^1(\mathcal{L}_a)$ if and only if $V_\rho (\{T_t^a\}_{t>0})(f)\in L^1(\mathbb{R}^n)$. Furthermore, if $f\in L^1(\mathbb{R}^n)$ then
$$
\|f\|_{H^1(\mathcal{L}_a)}\sim \|f\|_{L^1(\mathbb{R}^n)}+\|V_\rho (\{T_t^a\}_{t>0})(f)\|_{L^1(\mathbb{R}^n)}.
$$
\end{thm}

As it was mentioned the maximal operator $T_*^a$ is bounded from $L^p(\mathbb{R}^n)$ into itself, for every $1<p<\infty$, and from $L^1(\mathbb{R}^n)$ into $L^{1,\infty }(\mathbb{R}^n)$, provided that $a\geq 0$. We now complete these results. For every $\alpha \geq 0$ we define the maximal operator
$$
T_{*,\alpha}^a(f)=\sup_{t>0}|t^\alpha \partial _t^\alpha T_t^a(f)|.
$$
We remark that \cite[Corollary 4.2]{LeMX2} does not apply in this case.
\begin{thm}\label{Th1.3}
Let $\alpha \geq 0$. The operator $T_{*,\alpha}^a$ is bounded from $L^p(\mathbb{R}^n,w)$ into itself when one of the following conditions holds.

(i) $1<p<\infty$, $w\in A_p(\mathbb{R}^n)\cap RH_1(\mathbb{R}^n)$ and $a\geq 0$;

(ii) $\frac{n}{n-\sigma}<p<\frac{n}{\sigma}$, $w\in A_{(n-\sigma)p/n}(\mathbb{R}^n)\cap RH_{(\frac{n}{\sigma p})'}(\mathbb{R}^n)$ and $-\frac{(n-2)^2}{4}<a<0$.

Furthermore, $T_{*,\alpha}^a$ is bounded from $L^1(\mathbb{R}^n)$ into $L^{1,\infty }(\mathbb{R}^n)$ when $a\geq 0$. 
\end{thm}

In the following sections we prove Theorems \ref{Th1.1}, \ref{Th1.2} and \ref{Th1.3}. In order to establish weighted inequalities in Theorems \ref{Th1.1} and \ref{Th1.3} due to the lack of regularity conditions of the heat kernel associated with $\mathcal{L}_a$ we can not apply Calder\'on-Zygmund theory even when $a\geq 0$. We will use \cite[Theorem 6.6]{BZ} and we need to prove unweighted $L^p$-boundedness properties before establishing weighted inequalities.

Throughout the paper by $c$ and $C$ we denote positive constants that can change in each occurrence.

\section{Proof of Theorem 1.1}

We first prove the unweighted $L^p$-boundedness properties for the variation operator \newline  $V_\rho(\{t^\alpha\partial_t^\alpha T_t^a\}_{t>0})$. The semigroup of operators $\{T_t^a\}_{t>0}$, when $a\not=0$, is not conservative. Indeed, since $\sigma<\frac{n-2}{2}$, (\ref{1.1}) implies that $\displaystyle\int_{\mathbb R^n}T_t^a(x,y)dy<\infty$, $t>0$ and $x\in\mathbb R^n\setminus\{0\}$, and that $L_a(T_t^a(1)(x))=\partial_ tT_t^a(1)(x)$, $t>0$ and $x\in\mathbb R^n\setminus\{0\}$. Since $L_a(1)(x)=\frac{a}{|x|^2}$, $x\in\mathbb R^n\setminus\{0\}$, it follows that $T_t^a(1)\neq 1$, $t>0$. Moreover, $\{T_t^a\}_{t>0}$ is not contractive in $L^p(\mathbb R^n)$, $1<p<\infty$. Hence, $\{T_t^a\}_{t>0}$ is not a difussion semigroup in the Stein's sense (\cite{StLP}). Then, \cite[Theorem 3.3]{JR}  and \cite[Corollary 4.5]{LeMX} do not apply to $\{T_t^a\}_{t>0}$.

We denote by $\{W_t\}_{t>0}$ the Euclidean heat semigroup in $\mathbb R^n$, that is,
$$W_t(f)(x)=\int_{\mathbb R^n} W_t(x-y)f(y)dy,\;\;\;\;x\in\mathbb R^n,\;\mbox{and}\;t>0,$$
where $\displaystyle W_t(z)=\frac{1}{(4\pi t)^{n/2}}e^{-|z|^2/4t}$, $z\in\mathbb R^n$ and $t>0$.

In \cite[p. 12]{AB} it was proved that the operator $V_\rho(\{t^\alpha\partial_t^\alpha W_t\}_{t>0})$ is bounded from $L^p(\mathbb R^n)$ into itself, for every $1<p<\infty$, and from $L^1(\mathbb R^n)$ into $L^{1,\infty}(\mathbb R^n)$.

We are going to consider the operator
$$
S_{\rho,\alpha}=V_\rho(\{t^\alpha\partial_t^\alpha(T_t^a-W_t)\}_{t>0}).$$
Let $m\in \mathbb{N}$ such that $m-1\leq \alpha <m$. If $0<t_k<t_{k-1}< ... <t_1<\infty$, with $k\in\mathbb N$,  we can write
\begin{align*}
 \left(\sum_{j=1}^{k-1}\Big|t^\alpha\partial_t^\alpha(T_t^a-W_t)(f)(x)_{|t=t_j}-t^\alpha\partial_t^\alpha(T_t^a-W_t)(f)(x)_{|t=t_{j+1}}\Big|^\rho\right)^{1/\rho}& \\
& \hspace{-10cm}= \left(\sum_{j=1}^{k-1}\left|\int_{t_{j+1}}^{t_j}\partial_t(t^\alpha\partial_t^\alpha(T_t^a-W_t)(f)(x))dt\right|^\rho\right)^{1/\rho} \leq \sum_{j=1}^{k-1}\int_{t_{j+1}}^{t_j}|\partial_t(t^\alpha\partial_t^\alpha(T_t^a-W_t)(f)(x))|dt\\
& \hspace{-10cm}\leq \int_0^\infty|\partial_t(t^\alpha\partial_t^\alpha(T_t^a-W_t)(f)(x))|dt,\quad x \in\mathbb R^n.
\end{align*}
By proceeding as in the proof of \cite[Theorem 2.3]{TZ} we can write that
$$\partial_t\partial_t^\alpha(T_t^a-W_t)(f)(x)=\partial_t^{\alpha +1}(T_t^a-W_t)(f)(x),\;\;\;\;x\in\mathbb R^n\;\mbox{and}\;t>0.$$
Then, we get
$$S_{\rho,\alpha}(f)(x)\leq \alpha\int_0^\infty|t^{\alpha-1}\partial_t^\alpha(T_t^a-W_t)(f)(x)|dt+ \int_0^\infty|t^\alpha\partial_t^{\alpha+1}(T_t^a-W_t)(f)(x)|dt,\;\;\;\;x\in\mathbb R^n.$$
Let $\beta >0$ and $k\in\mathbb N$ such that $k-1\leq\beta <k$. We obtain
\begin{align}\label{2.0}
\int_0^\infty t^{\beta -1}|\partial_t^\beta(T_t^a-W_t)(f)(x)|dt& \nonumber\\
&\hspace{-3cm} =\frac{1}{\Gamma(k-\beta )}\int_0^\infty t^{\beta -1}\Big|\int_t^\infty\partial_u^k(T_u^a-W_u)(f)(x)(u-t)^{k-\beta -1}du\Big|dt \nonumber \\ 
&\hspace{-3cm} \leq C\int_0^\infty|\partial_u^k(T_u^a-W_u)(f)(x))|\int_0^u (u-t)^{k-\beta -1}t^{\beta -1}dtdu \nonumber\\
&\hspace{-3cm} \leq C\int_0^\infty u^{k-1}|\partial_u^k(T_u^a-W_u)(f)(x))|du,\;\;\;\;x\in\mathbb R^n.
\end{align}

For every $k\in\mathbb N$, $k\geq 1$, we define
$$\mathbb S_k(f)(x)=\int_{\mathbb R^n} f(y) \mathbb{K}_k(x,y)dy,\quad x\in\mathbb R^n,$$
where
$$\mathbb{K}_k(x,y)=\int_0^\infty u^{k-1}|\partial_u^k(T_u^a(x,y)-W_u(x-y))|du,\quad x,y\in\mathbb R^n.$$
We have that
$$S_{\rho,\alpha}(f)\leq C(\mathbb S_m(f)+\mathbb S_{m+1}(f)).$$

Our next objective is to establish the $L^p$-boundedness properties for $\mathbb S_k$, $k\in\mathbb N$, $k\geq 1$.

Let $k\in\mathbb N$, $k\geq 1$. We decompose $\mathbb{K}_k$ as follows
\begin{align*}
\mathbb{K}_k(x,y &)= \mathcal X_{A_1}(x,y) \mathbb{K}_k(x,y)+  \mathcal X_{A_2}(x,y) \mathbb{K}_k(x,y) \\
& \quad + \mathcal X_{A_3}(x,y)\int_{|x|^2}^\infty u^{k-1}|\partial_u^k(T_u^a(x,y)-W_u(x-y))|du \\
& \quad + \mathcal{X}_{A_3}(x,y)\int_0^{|x|^2} u^{k-1}|\partial_u^k(T_u^a(x,y)-W_u(x-y))|du \\
& =\sum_{j=1}^4 \mathbb{K}_{k,j}(x,y),\;\;\;\;x,y\in\mathbb R^n\setminus\{0\},
\end{align*}
being $A_1=\{(x,y)\in\mathbb R^n:\;0<|y|<\frac{|x|}{2}\}$, $A_2=\{(x,y)\in\mathbb R^n:\;|y|>\frac{3|x|}{2}>0\}$ and $A_3=\{(x,y)\in\mathbb R^n:\;0<\frac{|x|}{2}\leq |y|\leq \frac{3|x|}{2}\}$.

We now estimate $\mathbb{K}_{k,j}$, $j=1,2,3,4$.

\noindent {\bf (I)} We have that
\begin{align*}
\mathbb{K}_{k,1}(x,y)&\leq  \mathcal X_{A_1}(x,y) \int_0^\infty u^{k-1}|\partial_u^k W_u(x-y)|du+\mathcal X_{A_1}(x,y) \int_0^\infty u^{k-1}|\partial_u^k T_u^a(x,y)|du \\
&= \mathbb{K}_{k,1,1}(x,y)+\mathbb{K}_{k,1,2}(x,y),\;\;\;\;x,y\in\mathbb R^n.
\end{align*}
Since
\begin{align}\label{2.1}
|\partial_u^k W_u(z)|\leq C\frac{e^{-c\frac{|z|^2}{u}}}{u^{\frac{n}{2}+k}},\;\;\;\;z\in\mathbb R^n\;\mbox{and}\;u>0,
\end{align}
we get
\begin{align*}
\mathbb{K}_{k,1,1}(x,y)\leq & C\mathcal X_{A_1}(x,y) \int_0^\infty \frac{e^{-c\frac{|x-y|^2}{u}}}{u^{\frac{n}{2}+1}}du\leq \frac{C}{|x-y|^n}\mathcal X_{A_1}(x,y)\leq \frac{C}{|x|^n}\mathcal X_{A_1}(x,y),\;\;\;\;x,y\in\mathbb R^n.
\end{align*}
According to (\ref{1.3}) we get
\begin{align*}
\mathbb{K}_{k,1,2}(x,y)&\leq  C\mathcal X_{A_1}(x,y) \int_0^\infty\left(1+\frac{\sqrt{u}}{|x|}\right)^\sigma\left(1+\frac{\sqrt{u}}{|y|}\right)^\sigma \frac{e^{-c\frac{|x-y|^2}{u}}}{u^{\frac{n}{2}+1}}du \\
& \leq C\mathcal X_{A_1}(x,y)\left\{\begin{array}{lll} 
\frac{1}{|x|^n}, & \sigma\leq 0 & \\
 & &\\
 \frac{1}{|x|^n}+\frac{1}{|x|^{n-\sigma}|y|^\sigma}, & 0<\sigma<\frac{n-2}{2} &  \end{array}\right. ,\quad x,y\in\mathbb R^n. 
\end{align*}

We conclude that
\begin{equation}\label{2.2}
\mathbb{K}_{k,1}(x,y)\leq C\mathcal X_{A_1}(x,y)\left\{\begin{array}{lll} 
\frac{1}{|x|^n}, & \sigma\leq 0 & \\
 & &\\
 \frac{1}{|x|^n}+\frac{1}{|x|^{n-\sigma}|y|^\sigma}, & 0<\sigma<\frac{n-2}{2} &  \end{array}\right. ,\quad x,y\in\mathbb R^n. 
\end{equation}

\noindent {\bf (II)} By proceeding as in the proof of (\ref{2.2}) we can obtain
\begin{equation}\label{2.3}
\mathbb{K}_{k,2}(x,y)\leq  C\mathcal X_{A_2}(x,y)\left\{\begin{array}{lll} 
\frac{1}{|y|^n}, & \sigma\leq 0 & \\
 & & \\
 \frac{1}{|y|^n}+\frac{1}{|y|^{n-\sigma}|x|^\sigma}, & 0<\sigma<\frac{n-2}{2} &  \end{array}\right.,\quad x,y\in\mathbb R^n. 
 \end{equation}

\noindent {\bf (III)} We have that
\begin{align*}
\mathbb{K}_{k,3}(x,y)&\leq  \mathcal X_{A_3}(x,y) \left(\int_{|x|^2}^\infty u^{k-1}|\partial_u^k W_u(x-y)|du+\int_{|x|^2}^\infty u^{k-1}|\partial_u^k T_u^a(x,y)|du\right) \\
& = \mathbb{K}_{k,3,1}(x,y)+\mathbb{K}_{k,3,2}(x,y),\;\;\;\;x,y\in\mathbb R^n.
\end{align*}
From (\ref{2.1}) we deduce
\begin{align*}
\mathbb{K}_{k,3,1}(x,y)\leq & C\mathcal X_{A_3}(x,y) \int_{|x|^2}^\infty \frac{du}{u^{\frac{n}{2}+1}}\leq \frac{C}{|x|^n}\mathcal X_{A_3}(x,y) ,\;\;\;\;x,y\in\mathbb R^n.
\end{align*}
By using (\ref{1.3}) we get
\begin{align*}
\mathbb{K}_{k,3,2}(x,y)&\leq  C\mathcal X_{A_3}(x,y) \int_{|x|^2}^\infty\left(1+\frac{\sqrt{u}}{|x|}\right)^\sigma\left(1+\frac{\sqrt{u}}{|y|}\right)^\sigma \frac{e^{-c\frac{|x-y|^2}{u}}}{u^{\frac{n}{2}+1}}du \\
& \leq C\mathcal X_{A_3}(x,y)\left\{\begin{array}{ll} 
\displaystyle \int_{|x|^2}^\infty \frac{du}{u^{\frac{n}{2}+1}}, & \sigma\leq 0,  \\
 &  \\
\displaystyle \frac{1}{|x|^{2\sigma}}\int_{|x|^2}^\infty \frac{du}{u^{\frac{n}{2}-\sigma +1}}, & 0<\sigma<\frac{n-2}{2}, 
\end{array}\right. \\
& \leq \frac{C}{|x|^n}\mathcal X_{A_3}(x,y),\;\;\;\;x,y\in\mathbb R^n.
\end{align*}
Hence,
\begin{align}\label{2.4}
\mathbb{K}_{k,3}(x,y)\leq & \frac{C}{|x|^n}\mathcal X_{A_3}(x,y)
,\quad x,y\in\mathbb R^n. 
\end{align}

\noindent {\bf (IV)} According to Duhamel formula we get, for each $x,y\in\mathbb R^n$ and $t>0$,
\begin{align*}
& T_t^a(x,y)-W_t(x-y)=-a\int_0^t\int_{\mathbb R^n} W_{t-s}(x-z)|z|^{-2}T_s^a(z,y)dz ds \\
& =-a\int_0^{t/2}\int_{\mathbb R^n}W_{t-s}(x-z)|z|^{-2}T_s^a(z,y)dzds-a\int_0^{t/2}\int_{\mathbb R^n}W_s(x-z)|z|^{-2}T_{t-s}^a(z,y)dzds.
\end{align*}
It follows that
\begin{align*}
|\partial_t^k(T_t^a(x,y)-W_t(x-y))|&\leq C\left(\sum_{j=0}^{k-1}\int_{\mathbb R^n}|\partial_t^j(W_{\frac{t}{2}}(x-z))||\partial_t^{k-1-j}(T_{t/2}^a(z,y))||z|^{-2}dz \right. \\
&\quad +\int_0^{t/2}\int_{\mathbb R^n}|\partial_u^k W_u(x-z)_{|u=t-s}||T_s^a(z,y)||z|^{-2}dzds\\
&\left.\quad +\int_0^{t/2}\int_{\mathbb R^n}| W_s(x-z)||\partial_u^k T_u^a(z,y)|_{|u=t-s}|z|^{-2}dzds\right) \\
& =\sum_{j=0}^{k+1}H_j(t,x,y),\;x,y\in\mathbb R^n\;\mbox{and}\;t>0,
\end{align*}
where, for every $x,y\in\mathbb R^n$ and $t>0$,
$$H_j(t,x,y)=\int_{\mathbb R^n}|\partial_t^j(W_{\frac{t}{2}}(x-z))||\partial_t^{k-1-j}(T_{t/2}^a(z,y))||z|^{-2}dz,\;\;\;j=0,1,...,k-1,$$
$$H_k(t,x,y)=\int_0^{t/2}\int_{\mathbb R^n}|\partial_u^k W_u(x-z)_{|u=t-s}||T_s^a(z,y)||z|^{-2}dzds,$$
and
$$H_{k+1}(t,x,y)=\int_0^{t/2}\int_{\mathbb R^n}| W_s(x-z)||\partial_u^k T_u^a(z,y)|_{|u=t-s}|z|^{-2}dzds.$$
Assume that $a\geq 0$. Then $\sigma\leq 0$. Let $j=0,1,...,k-1$. According to (\ref{1.3}) and (\ref{2.1}) we get
$$H_j(t,x,y)\leq C\int_{\mathbb R^n}\frac{e^{-c\frac{|x-z|^2}{t}}e^{-c\frac{|z-y|^2}{t}}}{t^{n+k-1}}|z|^{-2}dz,\;\;\;x,y\in\mathbb R^n\;\mbox{and}\; t>0.$$
By using \cite[(28)]{BADLL} we obtain
$$H_j(t,x,y)\leq C\frac{e^{-c\frac{|x-y|^2}{t}}}{|x|^2t^{n/2+k-1}},\;\;\;\;x,y\in\mathbb R^n\;\mbox{and}\;t>0.$$
Since $n\geq 3$, it follows that
\begin{align}\label{2.5}
& \mathcal X_{A_3}(x,y)\int_0^{|x|^2} t^{k-1}H_j(t,x,y)dt\leq \frac{C}{|x|^2}\int_0^{|x|^2}\frac{e^{-c\frac{|x-y|^2}{t}}}{t^{n/2}}dt \nonumber \\
& \leq \frac{C}{|x|^2|x-y|^{n-2}}\int_{\frac{|x-y|^2}{|x|^2}}^\infty e^{-cu} u^{n/2-2}du \nonumber \\
& \leq \frac{C}{|x|^2|x-y|^{n-2}},\;\;\;\;x,y\in\mathbb R^n,\;x\neq y. 
\end{align}

Since $\frac{|x-z|^2}{t-s}+\frac{|z-y|^2}{s}\geq c\frac{|x-y|^2}{t}$, $0<s<t$, (\ref{1.3}) and (\ref{2.1}) lead to
\begin{align*}
H_k(t,x,y) & \leq C \int_0^{t/2}\int_{\mathbb R^n}\frac{e^{-c\frac{|x-z|^2}{t-s}}}{(t-s)^{n/2+k}}\frac{e^{-c\frac{|z-y|^2}{s}}}{s^{n/2}}|z|^{-2}dzds,\\
&\leq C \frac{e^{-c\frac{|x-y|^2}{t}}}{t^{n/2+k}}\int_0^{t/2}\int_{\mathbb R^n}\frac{e^{-c\frac{|z-y|^2}{s}}}{s^{n/2}}|z|^{-2}dzds.
\end{align*}
By using \cite[(28)]{BADLL} we obtain
$$H_k(t,x,y)\leq C\frac{e^{-c\frac{|x-y|^2}{t}}}{t^{n/2+k-1}}\frac{1}{|y|^2},\;\;\;\;x,y\in\mathbb R^n,\;\mbox{and}\;t>0.$$
Then,
\begin{align}\label{2.6}
 \mathcal X_{A_3}(x,y)\int_0^{|x|^2} t^{k-1}H_k(t,x,y)dt\leq \frac{C}{|x|^2|x-y|^{n-2}},\;\;\;\;x,y\in\mathbb R^n,\;x\neq y. 
\end{align}
In a similar way we get
\begin{align}\label{2.7}
 \mathcal X_{A_3}(x,y)\int_0^{|x|^2} t^{k-1}H_{k+1}(t,x,y)dt\leq \frac{C}{|x|^2|x-y|^{n-2}},\;\;\;\;x,y\in\mathbb R^n,\;x\neq y. 
\end{align}
By putting together (\ref{2.5}), (\ref{2.6}) and (\ref{2.7}) we conclude that, when $a\geq 0$,
\begin{align}\label{2.8}
 \mathbb{K}_{k,4}(x,y)\leq C\mathcal X_{A_3}(x,y)\frac{1}{|x|^2|x-y|^{n-2}},\;\;\;\;x,y\in\mathbb R^n,\;x\neq y. 
\end{align}

Assume now that $-\frac{(n-2)^2}{4}<a<0$. Then, $0<\sigma<\frac{n-2}{2}$. Let $j=0,1,...,k-1$. By using again (\ref{1.3}) and (\ref{2.1}) and by taking into account that, since $\sigma +2 <n$,
\begin{align}\label{2.9}
\int_{\mathbb R^n}\frac{e^{-c\frac{|x-z|^2}{t}}}{t^{n/2}}\left(1+\frac{\sqrt{s}}{|z|}\right)^\sigma |z|^{-2}dz\leq\frac{C}{|x|^2},\;\;\;\;0<s\leq t<|x|^2\;\mbox{and}\;x\in\mathbb R^n, 
\end{align}
we obtain
\begin{align*}
H_j(t,x,y) & \leq C \int_{\mathbb R^n}\frac{e^{-c\frac{|x-z|^2}{t}}e^{-c\frac{|z-y|^2}{t}}}{t^{n+k-1}}\left(1+\frac{\sqrt{t}}{|z|}\right)^\sigma \left(1+\frac{\sqrt{t}}{|y|}\right)^\sigma |z|^{-2}dz \\
&\leq C \frac{e^{-c\frac{|x-y|^2}{t}}}{t^{n/2+k-1}}\left(1+\frac{\sqrt{t}}{|y|}\right)^\sigma\frac{1}{|x|^2} \\
& \leq C \frac{e^{-c\frac{|x-y|^2}{t}}}{t^{n/2+k-1}}\frac{1}{|x|^2},\;\;\;\;(x,y)\in A_3,\;0<t<|x|^2.
\end{align*}
Since $n\geq 3$ it follows that
\begin{align}\label{2.10}
\mathcal X_{A_3}(x,y)\int_0^{|x|^2} t^{k-1}H_{j}(t,x,y)dt\leq \frac{C}{|x|^2|x-y|^{n-2}},\;\;\;\;x,y\in\mathbb R^n,\;x\neq y. 
\end{align}

On the other hand, according to (\ref{2.9}) we can write
\begin{align*}
H_k(t,x,y) & \leq C \int_0^{t/2}\int_{\mathbb R^n}\frac{e^{-c\frac{|x-z|^2}{t-s}}}{(t-s)^{n/2+k}}\frac{e^{-c\frac{|z-y|^2}{s}}}{s^{n/2}}\left(1+\frac{\sqrt{s}}{|z|}\right)^\sigma \left(1+\frac{\sqrt{s}}{|y|}\right)^\sigma |z|^{-2}dzds \\
&\leq C \frac{e^{-c\frac{|x-y|^2}{t}}}{t^{n/2+k}}\int_0^{t/2}\int_{\mathbb R^n}\frac{e^{-c\frac{|y-z|^2}{s}}}{s^{n/2}}\left(1+\frac{\sqrt{s}}{|z|}\right)^\sigma |z|^{-2}dzds \\
& \leq C \frac{e^{-c\frac{|x-y|^2}{t}}}{t^{n/2+k-1}}\frac{1}{|x|^2},\;\;\;\;(x,y)\in A_3,\;0<t<|x|^2.
\end{align*}
As in (\ref{2.6}) we get
\begin{align}\label{2.11}
\mathcal X_{A_3}(x,y)\int_0^{|x|^2} t^{k-1}H_{k}(t,x,y)dt\leq \frac{C}{|x|^2|x-y|^{n-2}},\;\;\;\;x,y\in\mathbb R^n,\;x\neq y. 
\end{align}
In a similar way we obtain
\begin{align}\label{2.12}
\mathcal X_{A_3}(x,y)\int_0^{|x|^2} t^{k-1}H_{k+1}(t,x,y)dt\leq \frac{C}{|x|^2|x-y|^{n-2}},\;\;\;\;x,y\in\mathbb R^n,\;x\neq y. 
\end{align}
From (\ref{2.10}), (\ref{2.11}) and (\ref{2.12}) we deduce that, when $-\frac{(n-2)^2}{4}<a<0$,
\begin{align}\label{2.13}
 \mathbb{K}_{k,4}(x,y)\leq C\mathcal X_{A_3}(x,y)\frac{1}{|x|^2|x-y|^{n-2}},\;\;\;\;x,y\in\mathbb R^n,\;x\neq y. 
\end{align}

We now consider the following operators

\begin{align*}
& J_1(f)(x)=\frac{1}{|x|^n}\int_{\mathbb R^n}\mathcal X_{A_1}(x,y)f(y)dy,\;\;\;\;x\in\mathbb R^n;\\
& J_2^\sigma(f)(x)=\frac{1}{|x|^{n-\sigma}}\int_{\mathbb R^n}\frac{\mathcal X_{A_1}(x,y)}{|y|^\sigma}f(y)dy,\;\;\;\;x\in\mathbb R^n\;\mbox{and}\;0<\sigma<\frac{n-2}{2};\\
& J_3(f)(x)=\int_{\mathbb R^n}\frac{\mathcal X_{A_2}(x,y)}{|y|^n}f(y)dy,\;\;\;\;x\in\mathbb R^n;\\
& J_4^\sigma(f)(x)=\frac{1}{|x|^{\sigma}}\int_{\mathbb R^n}\frac{\mathcal X_{A_2}(x,y)}{|y|^{n-\sigma}}f(y)dy,\;\;\;\;x\in\mathbb R^n\;\mbox{and}\;0<\sigma<\frac{n-2}{2};\\
& J_5(f)(x)=\frac{1}{|x|^n}\int_{\mathbb R^n}\mathcal X_{A_3}(x,y)f(y)dy,\;\;\;\;x\in\mathbb R^n;\\
& J_6(f)(x)=\frac{1}{|x|^2}\int_{\mathbb R^n}\frac{\mathcal X_{A_3}(x,y)}{|x-y|^{n-2}}f(y)dy,\;\;\;\;x\in\mathbb R^n.
\end{align*}
We have that
$$|J_j(f)|\leq C\mathcal M(f),\;\;\;\;j=1,5,$$
where $\mathcal M$ denotes the Hardy-Littlewood  maximal operator. Then $J_1$ and $J_5$ are bounded operators from $L^p(\mathbb R^n)$ into itself, for every $1<p\leq \infty$, and from $L^1(\mathbb R^n)$ into $L^{1,\infty}(\mathbb R^n)$.

For every $f\in L^1(\mathbb R^n)$ we can write
\begin{align*}
\int_{\mathbb R^n}|J_3(f)(x)|dx & \leq \int_{\mathbb R^n}\int_{\mathbb R^n}\mathcal X_{A_2}(x,y)|f(y)|\frac{dy}{|y|^n}dx \\
&\leq C \int_{\mathbb R^n}|f(y)|\frac{1}{|y|^n}\int_{B(0,|y|)}dxdy\leq C\|f\|_{L^1(\mathbb R^n)}.
\end{align*}

On the other hand, by using duality between $J_3$ and $J_1$ we deduce that $J_3$ is bounded from $L^p(\mathbb R^n)$ into itself, for every $1< p <\infty$.

We have that
\begin{align*}
\frac{1}{|x|^2} & \int_{\mathbb R^n}\frac{\mathcal X_{A_3}(x,y)}{|x-y|^{n-2}}dy=\frac{1}{|x|^2} \int_{B(0,\frac{3|x|}{2})\setminus B(0,\frac{|x|}{2})}\frac{dy}{|x-y|^{n-2}}\leq \frac{1}{|x|^2} \int_{B(0,3|x|)}\frac{dy}{|y|^{n-2}} \\
& \leq\frac{C}{|x|^2} \int_0^{3|x|}\rho d\rho\leq C,\;\;\;\;x\in\mathbb R^n,
\end{align*}
and, in a similar way we can get
\begin{align*}
\int_{\mathbb R^n}\frac{\mathcal X_{A_3}(x,y)}{|x|^2|x-y|^{n-2}}dx\leq C\int_{\mathbb R^n}\frac{\mathcal X_{A_3}(x,y)}{|y|^2|x-y|^{n-2}}dx\leq C,\;\;\;\;y\in\mathbb R^n.
\end{align*}
Then, the operator $J_6$ is bounded from $L^p(\mathbb R^n)$ into itself, for every $1\leq p\leq\infty$.

According to \cite[Theorem 2.1]{DHK} we deduce that the operator $J_2^\sigma$ is bounded from $L^p(\mathbb R^n)$ into itself provided that $\frac{n}{n-\sigma}<p<\infty$. By using duality we can prove that $J_4^\sigma$ is bounded from $L^p(\mathbb R^n)$ into itself for every $1<p<\frac{n}{\sigma}$.

By putting together (\ref{2.2}), (\ref{2.3}), (\ref{2.4}), (\ref{2.8}) and (\ref{2.13}) we conclude that the operator $\mathbb S_k$ is bounded from $L^p(\mathbb R^n)$ into itself, for every $1< p<\infty$, and from $L^1(\mathbb R^n)$ into $L^{1,\infty}(\mathbb R^n)$, when $a\geq 0$, and from $L^p(\mathbb R^n)$ into itself when $-\frac{(n-1)^2}{4}<a<0$ and $\frac{n}{n-\sigma}<p<\frac{n}{\sigma}$. Note that $\sigma<\frac{n-2}{2}$.

The $L^p$-boundedness properties that we have just proved for $\mathbb S_k$, $k\in\mathbb N$, are also satisfied for the operator $S_{\rho,\alpha}$.

As it was mentioned, in \cite[p. 12]{AB} it was proved that the operator $V_\rho(\{t^\alpha\partial_t^\alpha W_t\}_{t>0})$ is bounded from $L^p(\mathbb R^n)$ into itself, for every $1< p<\infty$, and from $L^1(\mathbb R^n)$ into $L^{1,\infty}(\mathbb R^n)$.

Since $V_\rho(\{t^\alpha\partial_t^\alpha T_t^a\}_{t>0})(f)\leq S_{\rho,\alpha}(f)+V_\rho(\{t^\alpha\partial_t^\alpha W_t\}_{t>0})(f)$, we conclude that $V_\rho(\{t^\alpha\partial_t^\alpha T_t^a\}_{t>0})$ is bounded from $L^p(\mathbb R^n)$ into itself, for every $1< p<\infty$, and from $L^1(\mathbb R^n)$ into $L^{1,\infty}(\mathbb R^n)$, when $a\geq 0$, and from $L^p(\mathbb R^n)$ into itself when $-\frac{(n-1)^2}{4}<a<0$ and $\frac{n}{n-\sigma}<p<\frac{n}{\sigma}$.

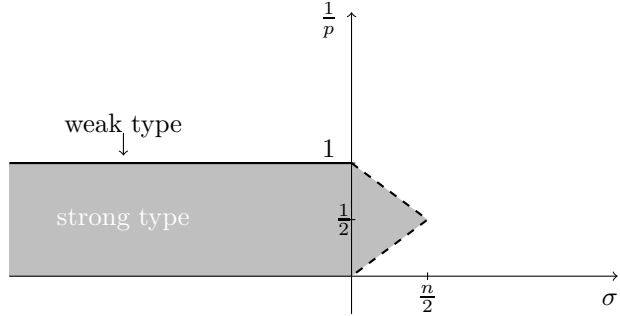
\begin{figure}[htbp]\label{Fig1}
\centering
\begin{minipage}{.5\textwidth}
  \centering
  \begin{tikzpicture}[scale=1]
    
    \draw (3.4,-0.3) node {$\sigma$};
    \draw (-0.3,3.4) node {$\frac{1}{p}$};
    
    \draw[-] (1,-0.05) -- (1,0.05);
  	\draw (1,-0.3) node {$\frac{n}{2}$};

    \draw[-] (-0.05, 1.5) -- (0.05,1.5);     
    \draw (-0.3,1.7) node {1};

	\fill[color=lightgray]
        (-4.5,1.5) -- (0,1.5) -- (0,0) -- (-4.5,0) -- (-4.5,1.5);
    \fill[color=lightgray]
        (0,1.5) -- (1,0.75) -- (0,0) -- (0,1.5);
     \draw[->] (0,-0.5) -- (0,3.5) ;
     
     	\draw[dashed, thick]  (0,1.5) -- (1,0.75) -- (0,0);  
	\draw[-, thick] (-4.5,1.5) -- (0,1.5);
	
	 \draw[->] (-4.5,0) -- (3.5,0) ;
        
       \draw (-3,0.75) node {\textcolor{white}{strong type}};
    \draw (-3,2) node {{weak type}};
    \draw[->] (-3,1.9) -- (-3,1.6);
     \draw[-] (-0.05, 0.75) -- (0.05,0.75);     
    \draw (-0.1,0.75) node{$\frac{1}{2}$};
 \end{tikzpicture}
  \label{fig:Fig1}
\end{minipage}
\end{figure}

\begin{rem}
This situation, as we can see in the graphic, remind us of the pencil phenomenom described in the Laguerre setting by \cite{MST} (see also \cite{NSj}).
\end{rem}

The arguments used in this part of the proof allow us to establish the following property.

\begin{prop}
The family of operators $\{\lambda(\Lambda(\{T_t^a\}_{t>0}),\lambda)^{1/2}\}_{\lambda >0}$ is uniformly bounded from $L^p(\mathbb R^n)$ into itself provided that $1< p<\infty$  when $a\geq 0$, and $\frac{n}{n-\sigma}<p<\frac{n}{\sigma}$ when $-\frac{(n-1)^2}{4}<a<0$.
\end{prop}
\begin{proof}
According to \cite[p. 6712]{JSW} we have that
$$\lambda[\Lambda(\{T_t^a-W_t\}_{t>0},\lambda)(f)]^{1/2}\leq CV_2(\{T_t^a-W_t\}_{t>0})(f),\;\;\;\;\lambda >0.$$
As in (\ref{2.0}) we get
$$V_2(\{T_t^a-W_t\}_{t>0})(f)(x)\leq\int_0^\infty |\partial_t(T_t^a-W_t)(f)(x)|dt,\;\;\;\;x\in\mathbb R^n.$$
By using \cite[(3)]{JSW} it follows that
\begin{align*}
 \lambda[(\Lambda(\{T_t^a\}_{t>0}),\lambda)(f)(x)]^{1/2} &\\
 &\hspace{-3cm}\leq C\lambda\left([(\Lambda(\{T_t^a-W_t\}_{t>0}),\frac{\lambda}{2})(f)(x)]^{1/2}+[(\Lambda(\{W_t\}_{t>0}),\frac{\lambda}{2})(f)(x)]^{1/2}\right) \\
&\hspace{-3cm} \leq\left(\int_0^\infty |\partial_t(T_t^a-W_t)(f)(x)|dt+\lambda[(\Lambda(\{W_t\}_{t>0}),\frac{\lambda}{2})(f)(x)]^{1/2} \right),\;\;\;\;x\in\mathbb R^n\;\mbox{and}\;\lambda>0.
\end{align*}
Since the family of operators $\{\lambda[\Lambda\{W_t\}_{t>0},\lambda)]^{1/2}\}_{\lambda >0}$ is uniformly bounded from $L^p(\mathbb R^n)$ into itself, for every $1<p<\infty$ (see \cite[Corollary, p. 7]{ZK}), and as it was showed above, the operator $\mathbb S_1$ is bounded from $L^p(\mathbb R^n)$ into itself for every $1<p<\infty$ when $a\geq 0$, and for every $\frac{n}{n-\sigma}<p<\frac{n}{\sigma}$ when $-\frac{(n-1)^2}{4}<a<0$, the proof of this proposition can be finished.
\end{proof}

\begin{rem}
We now state a negative result for the variation operator $V_\rho(\{T_t^a\}_{t>0})$ at the endpoint $p=\frac{n}{n-\sigma}$, when $0<\sigma<\frac{n-2}{2}$.
\end{rem}

\begin{prop}\label{Prop2.2}
Assume that $-\frac{(n-2)^2}{4}<a<0$. The operator $V_\rho(\{T_t^a\}_{t>0})$ is not bounded from $L^{\frac{n}{n-\sigma}}(\mathbb R^n)$ into $L^{\frac{n}{n-\sigma},\infty}(\mathbb R^n)$.
\end{prop}
\begin{proof}
Let $f\in L^{\frac{n}{n-\sigma}}(\mathbb R^n)\cap L^2(\mathbb R^n)$. We have that
$$|T_t^a(f)|\leq V_\rho(\{T_t^a\}_{t>0})(f)+|T_s^a(f)|,\;\;\;\;0<s<t<\infty.$$
Since $f\in L^2(\mathbb R^n)$, $\displaystyle \lim_{s\rightarrow 0^+}T_s^a(f)(x)=f(x)$, for almost all $x\in\mathbb R^n$ (See Appendix). Then, it follows that $T_*^a(f)\leq V_\rho(\{T_t^a\}_{t>0})(f)+|f|$, where $T_*^a$ represents the maximal operator $\displaystyle T_*^a(f)=\sup_{t>0}|T_t^a(f)|$.

Hence, if $V_\rho(\{T_t^a\}_{t>0})$ is bounded from $L^{\frac{n}{n-\sigma}}(\mathbb R^n)$ into $L^{\frac{n}{n-\sigma},\infty}(\mathbb R^n)$, there exists $C>0$ such that, for every $f\in L^{\frac{n}{n-\sigma}}(\mathbb R^n)\cap L^2(\mathbb R^n)$,
\begin{align}\label{2.14}
|\{x\in\mathbb R^n:\;T_*^a(f)(x)>\lambda\}|\leq C\left(\frac{\|f\|_{L^{\frac{n}{n-\sigma}}(\mathbb{R}^n)}}{\lambda}\right)^{\frac{n}{n-\sigma}},\quad\lambda>0.
\end{align}

\noindent We are going to see that the last property for $T_*^a$ does not hold.

Assume that $f\in L^2(\mathbb R^n)$, $f\geq 0$. According to (\ref{1.1}) we have that
\begin{align*}
 T_t^a(f)(x) & \geq C\left(1+\frac{\sqrt{t}}{|x|}\right)^\sigma\int_{\mathbb R^n}\left(1+\frac{\sqrt{t}}{|y|}\right)^\sigma\frac{e^{-c\frac{|x-y|^2}{t}}}{t^{n/2}}f(y)dy \\
& \geq  C\left(1+\frac{\sqrt{t}}{|x|}\right)^\sigma\int_{B(0,|x|)}\left(1+\frac{\sqrt{t}}{|y|}\right)^\sigma\frac{e^{-c\frac{|x-y|^2}{t}}}{t^{n/2}}f(y)dy,\;\;\;\;x\in\mathbb R^n\;\mbox{and}\;t>0.
\end{align*}
Then,
$$
 T_*^a(f)(x) \geq  T_{|x|^2}^a(f)(x)  \geq C\int_{B(0,|x|)}\left(\frac{|x|}{|y|}\right)^\sigma\frac{e^{-c\frac{|x-y|^2}{|x|^2}}}{|x|^n}f(y)dy,\;\;\;\;x\in\mathbb R^n.
$$
Since $|x-y|\leq 2|x|$, $y\in B(0,|x|)$, it follows that
$$T_*^a(f)(x)  \geq C\frac{1}{|x|^{n-\sigma}}\int_{B(0,|x|)}\frac{f(y)}{|y|^\sigma}dy,\;\;\;\;x\in\mathbb R^n.$$

Taking into account that $\displaystyle\int_{B(0,1)}|y|^{-\sigma\cdot\frac{n}{\sigma}}dy=\infty$, we can find, for every $k\in\mathbb N$, $f_k\geq 0$, $f_k\in L^{\frac{n}{n-\sigma}}(B(0,1))\cap L^2(B(0,1))$ such that $\|f_k\|_{L^{\frac{n}{n-\sigma}}(B(0,1))}\leq 1$ and $\displaystyle\int_{B(0,1)}f_k(x)\frac{dx}{|x|^\sigma}\geq k$. We define $f_k(x)=0$, $|x|\geq 1$.

We have that
$$T_*^a(f_k)(x)  \geq \frac{C}{|x|^{n-\sigma}}\int_{B(0,1)}\frac{f_k(y)}{|y|^\sigma}dy\geq C\frac{k}{|x|^{n-\sigma}},\;\;\;\;|x|\geq 1\;\mbox{and}\;k\in\mathbb N.$$
Then,
$$T_*^a(f_k)(x)  \geq C(k-1),\;\;1<|x|<\left(\frac{k}{k-1}\right)^{\frac{1}{n-\sigma}},\;\;k\in\mathbb N,\;k\geq 2.$$
We deduce that, for every $k\in\mathbb N,\;k\geq 2$,
\begin{align*}
\Big|\big\{x\in\mathbb R^n: & \;T_*^a(f_k)(x)>C\frac{k-1}{2}\big\}\Big|\geq \int_{1<|x|<\left(\frac{k}{k-1}\right)^{\frac{1}{n-\sigma}}}dx  =|B(0,1)|\left(\left(\frac{k}{k-1}\right)^{\frac{n}{n-\sigma}}-1\right).
\end{align*}
Since $\|f_k\|_{L^{\frac{n}{n-\sigma}}(\mathbb R^n)}\leq 1$ and $\{k^{\frac{n}{n-\sigma}}-(k-1)^{\frac{n}{n-\sigma}}\}_{k\geq 2}$ is not a bounded sequence we conclude that (\ref{2.14}) does not hold.
\end{proof}

Now we are going to study weighted $L^p$-boundedness properties for the operator $V_\rho(\{t^\alpha\partial_t^\alpha T_t^a\}_{t>0})$. We will use \cite[Theorem 6.6]{BZ}.

By proceeding as above we get, for every $x\in\mathbb R^n$,
\begin{align}\label{2.15}
V_\rho(\{t^\alpha\partial_t^\alpha T_t^a\}_{t>0})(f)(x) & \leq C\left(\int_0^\infty t^{m-1}|\partial_t^m T_t^a(f)(x)|dt\right. 
\left. +\int_0^\infty t^{m}|\partial_t^{m+1} T_t^a(f)(x)|dt\right),
\end{align}
where $m\in\mathbb N$ and $m-1\leq\alpha <m$.

We take $k\in\mathbb N$ that will be fixed later. Let us define, for every $t>0$,
$$A_t^a=I-(I-T_t^a)^k.$$
For every ball $B$ in $\mathbb R^n$ and $j\in\mathbb N$, $j\geq 1$, we consider $S_j(B)=2^j B\setminus 2^{j-1}B$, and we write $S_0(B)=B$.

Let $B$ be a ball in $\mathbb R^n$. By $r_B$ we denote the radius of $B$. Assume that $1<p\leq q<\infty$ when $a\geq 0$ and that $\frac{n}{n-\sigma}<p\leq q<\frac{n}{\sigma}$ when $-\frac{(n-2)^2}{4}<a<0$. Let $f$ be a smooth function supported in $B$ and let $j\in\mathbb N$, $j\geq 2$.

According to (\ref{1.1}) and \cite[Theorem 3.1]{BADLL} we deduce that
\begin{align}\label{2.16}
 \left(\frac{1}{|S_j(B)|}\right. & \left.\int_{S_j(B)}|A_{r_B^2}^a(f)(x)|^qdx\right)^{1/q}  \leq C\sum_{i=1}^k\left(\frac{1}{|S_j(B)|}\int_{S_j(B)}|T_{r_B^2 i}^a(f)(x)|^qdx\right)^{1/q} \nonumber \\
& \leq C\frac{e^{-c2^{2j}}}{|S_j(B)|^{1/q}}r_B^{-n(\frac{1}{p}-\frac{1}{q})}\|f\|_{L^p(\mathbb R^n)} \nonumber \\
& \leq C\frac{e^{-c2^{2j}}}{(2^jr_B)^{n/q}}|B|^{(\frac{1}{q}-\frac{1}{p})}\|f\|_{L^p(\mathbb R^n)} \nonumber \\
& \leq C\frac{e^{-c2^{2j}}}{2^{jn/q}}\left(\frac{1}{|B|}\int_{B}|f(x)|^pdx\right)^{1/p}.
\end{align}
Let $\ell\in\mathbb N$, $\ell\geq 1$. We define the operator
$$S_{\ell ,a}(f)(x)=\int_0^\infty t^{\ell -1}|\partial_t^\ell T_t^a(f)(x)|dt,\;\;\;\;x\in\mathbb R^n.$$
By using Minkowski inequality we have that
\begin{align*}
 \left(\frac{1}{|S_j(B)|}\right. & \left.\int_{S_j(B)}|S_{\ell ,a}((I-A_{r_B^2}^a)(f))(x)|^qdx\right)^{1/q}  \\
& \leq \int_0^\infty\left(\frac{1}{|S_j(B)|}\int_{S_j(B)}|t^{\ell -1}\partial_t^\ell T_t^a((I-T_{r_B^2}^a)^k(f))(x)|^qdx\right)^{1/q}dt  \\
& \leq C\sum_{i=0}^k\int_0^{r_B^2}\left(\frac{1}{|S_j(B)|}\int_{S_j(B)}|t^{\ell -1}\partial_t^\ell T_{t+r_B^2i}^a(f)(x)|^qdx\right)^{1/q}dt \\
& + \int_{r_B^2}^\infty\left(\frac{1}{|S_j(B)|}\int_{S_j(B)}|t^{\ell -1}\partial_t^\ell T_t^a((I-T_{r_B^2}^a)^k(f))(x)|^qdx\right)^{1/q}dt.
\end{align*}
Let $i=0,1,...,k$. We consider
$$I_{j,i}=\int_0^{r_B^2}\left(\frac{1}{|S_j(B)|}\int_{S_j(B)}|t^{\ell -1}\partial_t^\ell T_{t+r_B^2i}^a(f)(x)|^qdx\right)^{1/q}dt.$$
By using (\ref{1.3}) and \cite[Theorem 3.1]{BADLL} we get
\begin{align*}
&  I_{j,i}=\int_0^{r_B^2}\left(\frac{1}{|S_j(B)|}\int_{S_j(B)}\left|\frac{t^{\ell -1}}{(t+r_B^2 i)^\ell}u^\ell\partial_u^\ell T_u^a(f)(x)_{|u=t+r_B^2 i}\right|^qdx\right)^{1/q}dt  \\
& \leq C\|f\|_{L^q(\mathbb R^n)}\int_0^{r_B^2}\frac{1}{t+r_B^2}\frac{1}{|S_j(B)|^{1/q}}e^{-c\frac{r_B^2 2^{2j}}{t+r_B^2}}dt  \\
& \leq C\|f\|_{L^q(\mathbb R^n)}\int_0^{r_B^2}\frac{1}{r_B^2 2^{2j}}\frac{1}{(2^j r_B)^{n/q}}\left(\frac{t+r_B^2}{r_B^2 2^{2j}}\right)^k dt  \\
& \leq C 2^{-2j(k+1+\frac{n}{2q})}\left(\frac{1}{|B|}\int_{B}|f(x)|^qdx\right)^{1/q}.
\end{align*}
By \cite[(18)]{BADLL}
$$(I-T_{r_B^2}^a)^k=\int_0^{r_B^2}\cdots\int_0^{r_B^2} (\partial_t^k T_t^a) _{|t=s_1+s_2+...+s_k}ds_1\cdots ds_k.$$
Using again (\ref{1.3}) and \cite[Theorem 3.1]{BADLL} we have that
\begin{align*}
&  \int_{r_B^2}^\infty\left(\frac{1}{|S_j(B)|}\int_{S_j(B)}|t^{\ell -1}\partial_t^\ell T_t^a((I-T_{r_B^2}^a)^k(f))(x)|^qdx\right)^{1/q}dt  \\
& \leq \int_{r_B^2}^\infty\int_0^{r_B^2}\cdots\int_0^{r_B^2}\left(\frac{1}{|S_j(B)|}\int_{S_j(B)}|t^{\ell -1}\partial_t^{\ell +k} T_{t+s_1+...+s_k}^a(f)(x)|^qdx\right)^{1/q}ds_1\cdots ds_k dt \\
& \leq C\|f\|_{L^q(\mathbb R^n)}\int_{r_B^2}^\infty\int_0^{r_B^2}\cdots\int_0^{r_B^2}\frac{1}{|S_j(B)|^{1/q}}\frac{t^{\ell -1}}{(t+s_1+...+s_k)^{\ell +k}}e^{-c\frac{r_B^2 2^{2j}}{t+s_1+...+s_k}} ds_1\cdots ds_k dt \\
& \leq C\int_{r_B^2}^\infty\int_0^{r_B^2}\cdots\int_0^{r_B^2}\frac{1}{2^{jn/q}}\frac{1}{t^{k+1}}e^{-c\frac{r_B^2 2^{2j}}{t}} ds_1\cdots ds_k dt\left(\frac{1}{|B|}\int_{B}|f(x)|^qdx\right)^{1/q} \\
& \leq C\frac{r_B^{2k}}{2^{jn/q}}\int_{r_B^2}^\infty\frac{1}{t^{k+1}}e^{-c\frac{r_B^2 2^{2j}}{t}}dt \left(\frac{1}{|B|}\int_{B}|f(x)|^qdx\right)^{1/q} \\
& \leq C 2^{-j(n/q+2k)}\left(\frac{1}{|B|}\int_{B}|f(x)|^qdx\right)^{1/q}.
\end{align*}
We conclude that
$$\left(\frac{1}{|S_j(B)|}\int_{S_j(B)}|S_{\ell ,a}((I-A_{r_B^2}^a)(f))(x)|^qdx\right)^{1/q}\leq C 2^{-j(n/q+2k)}\left(\frac{1}{|B|}\int_{B}|f(x)|^qdx\right)^{1/q}.$$
By using (\ref{2.15}) we deduce
\begin{align*}
\left(\frac{1}{|S_j(B)|}\int_{S_j(B)} \right. & \left. |V_\rho(\{t^\alpha\partial_t^\alpha T_t^a\}_{t>0})((I-A_{r_B^2}^a)(f))(x)|^qdx\right)^{1/q} \\
& \leq C 2^{-j(n/q+2k)}\left(\frac{1}{|B|}\int_{B}|f(x)|^qdx\right)^{1/q} .
\end{align*}
We choose $k\in\mathbb N$ such that $k>n/2$. Then, $\sum_{j\in \mathbb{N}} 2^{-j(n/q+2k)}2^{jn}<\infty$. As we have proved above, the operator $V_\rho(\{t^\alpha\partial_t^\alpha T_t^a\}_{t>0})$ is bounded from $L^p(\mathbb R^n)$ into itself, for every $1<p<\infty$, when $a\geq 0$, and for every $\frac{n}{n-\sigma}<p<\frac{n}{\sigma}$, when $-\frac{(n-2)^2}{4}<a<0$, and then from \cite[Theorem 6.6]{BZ} we deduce that $V_\rho(\{t^\alpha\partial_t^\alpha T_t^a\}_{t>0})$ is bounded from $L^p(\mathbb R^n,w)$ into itself provided that $1<p<\infty$ and $w\in A_p(\mathbb R^n)\cap RH_1(\mathbb R^n)$ when $a\geq 0$, and provided that $\frac{n}{n-\sigma}<p<\frac{n}{\sigma}$ and $w\in A_{(n-\sigma)p/n}(\mathbb R^n)\cap RH_{(\frac{n}{\sigma p})'}(\mathbb R^n)$, when $-\frac{(n-2)^2}{4}<a<0$.

It is known that (see \cite[p. 6712]{JSW})
$$\lambda[\Lambda(\{t^\alpha\partial_t^\alpha T_t^a\}_{t>0},\lambda)(f)(x)]^{1/rho}\leq C_qV_\rho(\{t^\alpha\partial_t^\alpha T_t^a\}_{t>0})(f)(x),\;\;\;x\in\mathbb R^n\:\mbox{and}\;\lambda >0.$$
From the $L^p$-boundedness properties of $V_\rho(\{t^\alpha\partial_t^\alpha T_t^a\}_{t>0})$ we deduce the corresponding ones for the family $\{\lambda(\Lambda[\{t^\alpha\partial_t^\alpha T_t^a\}_{t>0},\lambda]^{1/\rho}\}_{\lambda >0}$.

Finally, we have that
\begin{align*}
& \mathcal O(\{t^\alpha\partial_t^\alpha T_t^a\}_{t>0},\{t_j\}_{j\in\mathbb N})(f)(x) \\
& =\left(\sum_{j\in\mathbb{N}}\sup_{t_{j+1}\leq s_{j+1}<s_j\leq t_j}\big|t^\alpha\partial_t^\alpha T_t^a(f)(x)_{|t=s_j}-t^\alpha\partial_t^\alpha T_t^a(f)(x)_{|t=s_{j+1}}\big|^2\right)^{1/2} \\
& \leq \left(\sum_{j\in \mathbb{N}}\sup_{t_{j+1}\leq s_{j+1}<s_j\leq t_j}\left|\int_{s_{j+1}}^{s_j}\partial_t(t^\alpha\partial_t^\alpha T_t^a(f)(x))dt\right|^2\right)^{1/2} \\
& \leq \int_0^\infty|\partial_t(t^\alpha\partial_t^\alpha T_t^a(f)(x))|dt,\;\;\;\;x\in\mathbb R^n,
\end{align*}
and in a similar way we can get, for every $k\in\mathbb N$,
$$V_{k}(\{t^\alpha\partial_t^\alpha T_t^a\}_{t>0})(f)(x)\leq \int_{2^{-k}}^{2^{-k+1}}|\partial_t(t^\alpha\partial_t^\alpha T_t^a(f)(x))|dt,\;\;\;\;x\in\mathbb R^n,$$
and then
$$S_V(\{t^\alpha\partial_t^\alpha T_t^a\}_{t>0})(f)(x)\leq \int_0^\infty|\partial_t(t^\alpha\partial_t^\alpha T_t^a(f)(x))|dt,\;\;\;\;x\in\mathbb R^n.$$

By proceeding as in the proof of the $L^p$-boundedness properties for the variation operator $V_{\rho}(\{t^\alpha\partial_t^\alpha T_t^a\}_{t>0})$ we can get the corresponding ones for the oscilation $\mathcal O(\{t^\alpha\partial_t^\alpha T_t^a\}_{t>0},\{t_j\}_{j\in\mathbb N})$ and the short variation $S_V(\{t^\alpha\partial_t^\alpha T_t^a\}_{t>0})$ operators.

\section{Proof of Theorem \ref{Th1.2}}
We are going to see that the variation operator $V_\rho (\{t^\alpha \partial _t^\alpha T_t^a\}_{t>0})$ is bounded from $H^1(\mathcal{L}_a)$ into $L^1(\mathbb{R}^n)$. According to Theorem \ref{Th1.1} $V_\rho(\{t^\alpha \partial _t^\alpha T_t^a\}_{t>0})$ is bounded from $L^2(\mathbb{R}^n)$ into itself. Then, by \cite[Lemma 4.3]{HLMMY}, in order to get our objective it is sufficient to see that, by taking $M>1$, there exists $C>0$ such that for every $(2,M)$-atom $b$ we have that
\begin{equation}\label{3.0}
\|V_\rho (\{t^\alpha \partial _t^\alpha T_t^a\}_{t>0})(b)\|_{L^1(\mathbb{R}^n)}\leq C.
\end{equation}
Assume that $b$ is a $(2,M)$-atom associated with the ball $B$. Since $V_\rho (\{t^\alpha \partial _t^\alpha T_t^a\}_{t>0})$ is bounded on $L^2(\mathbb{R}^n)$ it follows that
\begin{align} \label{3.1}
    \|V_\rho (\{t^\alpha \partial _t^\alpha T_t^a\}_{t>0})(b)\|_{L^1(4B)} & \leq |4B|^{1/2}\|V_\rho (\{t^\alpha \partial _t^\alpha T_t^a\}_{t>0})(b)\|_{L^2(\mathbb{R}^n)} \nonumber \\
    & \leq C|B|^{1/2}\|b\|_{L^2(\mathbb{R}^n)}\leq C,
\end{align}
where $C>0$ does not depend on $b$.

By proceeding as in the proof of Theorem \ref{Th1.1} we can see that
\begin{align*}
   V_\rho (\{t^\alpha \partial _t^\alpha T_t^a\}_{t>0})(b)(x) & \leq \int_0^\infty |\partial _t(t^\alpha \partial _t^\alpha T_t^a(b)(x)|dt \\
   &\hspace{-1cm}\leq C\left(\mathcal{X}_{\{\gamma >0\}}(\alpha )\int_0^\infty |t^{\alpha -1}\partial _t^\alpha T_t^a(b)(x)|dt +\int_0^\infty |t^\alpha \partial _t^{\alpha +1} T_t^a(b)(x)|dt\right),\quad x\in \mathbb{R}^n.
\end{align*}

If $\beta >0$ and $k\in \mathbb{N}$ such that $k-1\leq \beta <k$, we have, as in \eqref{2.0}, that
$$
\int_0^\infty |t^{\beta -1}\partial _t^\beta T_t^a(b)(x)|dt\leq C\int_0^\infty |t^{k-1}\partial _t^kT_t^a(b)(x)|dt.
$$

Consider $k\in \mathbb{N}$, $k\geq 1$. We define the operator $S_{k,a}$ as follows
$$
S_{k,a}(f)(x)=\int_0^\infty |t^{k-1}\partial _t^k T_t^a(f)(x)|dt,\quad x\in \mathbb{R}^n.
$$

We are going to see that
$$
\|S_{k,a}(b)\|_{L^1(\mathbb{R}^n\setminus (4B))}\leq C,
$$
where $C$ does not depend on $b$.

There exists $u \in D(\mathcal{L}_a)$ such that $b=\mathcal{L}_au$, $\supp u\subset B$ and $\|u\|_{L^2(\mathbb{R}^n)}\leq r_B^2|B|^{-1/2}$. We have that $T_t^ab=T_t^a\mathcal{L}_au=\mathcal{L}_aT_t^au=\partial _tT_t^au$, $t>0$. From \eqref{1.3} and since $|x-y|\geq c|x-x_B|$, when $x\not \in 4B$ and $y\in B$, it follows that
\begin{align*}
    S_{k,a}(b)(x)&\leq C\int_0^\infty t^{k-1}|\partial _t^{k+1}T_t^a(u)(x)|dt\leq C\int_0^\infty\int_B\frac{e^{-c\frac{|x-y|^2}{t}}}{t^{\frac{n}{2}+2}}|u(y)|dydt\\
    &\leq C\int_0^\infty\int_B\frac{e^{-c\frac{|x-x_B|^2}{t}}}{t^{\frac{n}{2}+2}}|u(y)|dydt
    \leq C\frac{\|u\|_{L^1(\mathbb{R}^n)}}{|x-x_B|^{n+2}}\\
    &\leq C\frac{\|u\|_{L^2(\mathbb{R}^n)}|B|^{1/2}}{|x-x_B|^{n+2}}\leq C\frac{r_B^2}{|x-x_B|^{n+2}},\quad x\not \in 4B.
\end{align*}
Then,
\begin{equation}\label{3.5bis}
\|S_{k,a}(b)\|_{L^1(\mathbb{R}^n\setminus (4B))}\leq Cr_B^2 \int_{4r_B}^\infty \frac{d\rho }{\rho^3}\leq C,
\end{equation}
where $C$ does not depend on $b$.

By using \eqref{3.1} and \eqref{3.5bis} for $k=1$ when $\alpha =0$ and with $k=m$ and $k=m+1$, when $\alpha >0$ and $m\in \mathbb{N}$ such that $m-1\leq \alpha <m$, it follows that
$$
\|V_\rho (\{t^\alpha \partial_t^\alpha T_t^a\}_{t>0})(b)\|_{L^1(\mathbb{R}^n)}\leq C,
$$
being $C$ independent of $b$.

In order to see that the oscillations, jump and short variation operators are bounded from $H^1(\mathcal{L}_a)$ into $L^1(\mathbb{R}^n)$ we can proceed as in the end of the proof of Theorem \ref{Th1.1}.

For every $t>0$ we have that $|T_t^af|\leq |T_t^af-T_1^af|+|T_1^af|$. Then,
$$
|T_t^af|\leq V_\rho (\{T_t^a\}_{t>0})(f)+|T_1^af|.
$$
Since $a \geq 0$, from \eqref{1.1} we deduce that there exists $C>0$ such that
$$
\|T_t^af\|_{L^1(\mathbb{R}^n)}\leq C\|f\|_{L^1(\mathbb{R}^n)},\quad t>0\mbox{ and }f\in L^1(\mathbb{R}^n).
$$
Hence, if $f\in L^1(\mathbb{R}^n)$ and $V_\rho (\{T_t^a\}_{t>0})(f)\in L^1(\mathbb{R}^n)$, then $T_*^a(f)\in L^1(\mathbb{R}^n)$, and, thus, $f\in H^1(\mathcal{L}_a)$.

We conclude that if $f\in L^1(\mathbb{R}^n)$, then $f\in H^1(\mathcal{L}_a)$ if, and only if, $V_\rho (\{T_t^a\}_{t>0})(f)\in L^1(\mathbb{R}^n)$. Furthermore,
$$
\|f\|_{L^1(\mathbb{R}^n)}+\|T_*^af\|_{L^1(\mathbb{R}^n)}\sim \|f\|_{L^1(\mathbb{R}^n)}+\|V_\rho (\{T_t^a\}_{t>0})(f)\|_{L^1(\mathbb{R}^n)}.
$$

\section{Proof of Theorem \ref{Th1.3}}
We firstly establish the unweighted $L^p$-inequalities for $T_{*,\alpha}^a$. As it was mentioned in the introduction, by using \eqref{1.1} we can see that $T_{*,0}^a$ is bounded from $L^p(\mathbb{R}^n)$ into itself, for every $1<p<\infty$, and from $L^1(\mathbb{R}^n)$ into $L^{1,\infty}(\mathbb{R}^n)$, provided that $\sigma<0$, that is, $a\geq 0$.

Let $\alpha >0$. We choose $m\in \mathbb{N}$ such that $m-1\leq \alpha <m$. We have that
$$
\partial_t^\alpha T_t^a(f)(x)=\frac{1}{\Gamma (m-\alpha)}\int_0^\infty \partial _u^mT_u^a(f)(x)_{|u=t+s}s^{m-\alpha -1}ds,\quad x\in \mathbb{R}^n\mbox{ and }t>0.
$$
Then,
$$
|t^\alpha \partial_t^\alpha T_t^a(f)(x)|\leq Ct^\alpha T_{*,m}^a(f)(x)\int_0^\infty (t+s)^{-m}s^{m-\alpha -1}ds\leq CT_{*,m}^a(f)(x),\quad x\in \mathbb{R}^n,\;t>0.
$$
Hence, we get
$$
T_{*,\alpha }^a(f)(x)\leq CT_{*,m}^a(f)(x),\quad x\in \mathbb{R}^n.
$$

We consider the following maximal operator
$$
\mathcal{S}_{*,m}^a(f)(x)=\sup_{t>0}|t^m\partial_t^m(T_t^a-W_t)(f)(x)|,\quad x\in \mathbb{R}^n.
$$

We have that $\mathcal{S}_{*,m}^a(f)\leq \sum_{j=1}^4\mathcal{H}_j(f)$, where, for every $j=1,2,3,4$,
$$
\mathcal{H}_j(f)(x)=\int_{\mathbb{R}^n}H_j(x,y)f(y)dy,\quad x\in \mathbb{R}^n,
$$
being
\begin{align*}
    H_1(x,y)&=\mathcal{X}_{A_1}(x,y)\sup_{t>0}|t^m\partial _t^m(T_t^a(x,y)-W_t(x-y))|,\quad x,y\in \mathbb{R}^n,\\
    H_2(x,y)&=\mathcal{X}_{A_2}(x,y)\sup_{t>0}|t^m\partial _t^m(T_t^a(x,y)-W_t(x-y))|,\quad x,y\in \mathbb{R}^n,\\
    H_3(x,y)&=\mathcal{X}_{A_3}(x,y)\sup_{t\geq |x|^2}|t^m\partial _t^m(T_t^a(x,y)-W_t(x-y))|,\quad x,y\in \mathbb{R}^n,
\end{align*}
and
$$
H_4(x,y)=\mathcal{X}_{A_3}(x,y)\sup_{0<t\leq |x|^2}|t^m\partial _t^m(T_t^a(x,y)-W_t(x-y))|,\quad x,y\in \mathbb{R}^n.
$$
Here, as in the proof of Theorem \ref{Th1.1}, $A_1=\{(x,y)\in\mathbb R^n:\;0<|y|<\frac{|x|}{2}\}$, $A_2=\{(x,y)\in\mathbb R^n:\;|y|>\frac{3|x|}{2}>0\}$ and $A_3=\{(x,y)\in\mathbb R^n:\;0<\frac{|x|}{2}\leq |y|\leq \frac{3|x|}{2}\}$.

We estimate $H_j$, $j=1,2,3,4$.

(I) We have that
\begin{align*}
H_1(x,y)&\leq\mathcal{X}_{A_1}(x,y)\sup_{t>0}|t^m\partial _t^m T_t^a(x,y)|+\mathcal{X}_{A_1}(x,y)\sup_{t>0}|t^m\partial _t^m W_t(x-y)|\\
&=H_{1,1}(x,y)+H_{1,2}(x,y),\quad x,y\in \mathbb{R}^n.
\end{align*}
According to \eqref{2.1} we get
$$
H_{1,2}(x,y)\leq C\sup_{t>0}\frac{e^{-c\frac{|x-y|^2}{t}}}{t^{n/2}}\leq \frac{C}{|x-y|^n}\leq \frac{C}{|x|^n}, \quad (x,y)\in A_1.
$$
By using \eqref{1.3} we deduce
$$
H_{1,1}(x,y)\leq C\sup_{t>0}\left(1+\frac{\sqrt{t}}{|x|}\right)^\sigma \left(1+\frac{\sqrt{t}}{|y|}\right)^\sigma \frac{e^{-c\frac{|x-y|^2}{t}}}{t^{n/2}},\quad (x,y)\in A_1. 
$$
If $\sigma \leq 0$ then
\begin{equation}\label{4.1}
H_{1,1}(x,y)\leq \frac{C}{|x|^n},\quad (x,y)\in A_1.
\end{equation}
When $0<\sigma <\frac{n-2}{2}$ we can write
\begin{align*}
    H_{1,1}(x,y)&\leq C\sup_{t>0}\left(1+\frac{t^{\sigma /2}}{|x|^\sigma}+\frac{t^{\sigma /2}}{|y|^\sigma}+\frac{t^\sigma }{|x|^\sigma|y|^\sigma}\right)\frac{e^{-c\frac{|x|^2}{t}}}{t^{n/2}}\\
    &\leq C\left(\frac{1}{|x|^n}+\frac{1}{|x|^{n-\sigma}|y|^\sigma}\right),\quad (x,y)\in A_1.
\end{align*}

(II) In a similar way we can deduce that
$$
H_2(x,y)\leq C\left\{
\begin{array}{ll}
\displaystyle \frac{1}{|y|^n},&\sigma \leq 0,\\[0.5cm]
\displaystyle \frac{1}{|y|^n}+\frac{1}{|y|^{n-\sigma}|x|^\sigma },&\displaystyle 0<\sigma <\frac{n-2}{2}
\end{array}
\right. \quad  ,\; (x,y)\in A_2.
$$

(III) By using \eqref{1.3} we get
$$
H_3(x,y)\leq C\sup_{t\geq |x|^2}\left(1+\frac{\sqrt{t}}{|x|}\right)^\sigma \left(1+\frac{\sqrt{t}}{|y|}\right)^\sigma \frac{1}{t^{n/2}}\leq \frac{C}{|x|^n},\quad (x,y)\in A_3.
$$

(IV) By proceeding as in the part (IV) in the proof of Theorem \ref{Th1.1} we obtain
$$
H_4(x,y)\leq C\sup_{t\leq |x|^2}\frac{e^{-c\frac{|x-y|^2}{t}}}{|x|^2t^{n/2-1}}\leq \frac{C}{|x|^2|x-y|^{n-2}},\quad (x,y)\in A_3.
$$
The arguments in the proof of Theorem \ref{Th1.1} allow us to conclude that $\mathcal{S}_{*,m}^a$ is bounded from $L^p(\mathbb{R}^n)$ into itself, for every $1<p<\infty$, and from $L^1(\mathbb{R}^n)$ into $L^{1,\infty}(\mathbb{R}^n)$, when $a\geq 0$, and from $L^p(\mathbb{R}^n)$ into itself, for every $\frac{n}{n-\sigma}<p<\frac{n}{\sigma}$, when $-\frac{(n-2)^2}{4}<a<0$.

According to \eqref{2.1} the maximal operator $W_{*,m}$ defined by
$$
W_{*,m}(f)=\sup_{t>0}|t^m\partial _t^mW_t(f)|,
$$
is bounded from $L^p(\mathbb{R}^n)$ into itself, for every $1<p<\infty$, and from $L^1(\mathbb{R}^n)$ into $L^{1,\infty }(\mathbb{R}^n)$.

Then, we conclude that the operator $T_{*,m}^a$, and hence $T_{*,\alpha }^a$, is bounded from $L^p(\mathbb{R}^n)$ into itself, for every $1<p<\infty$, and from $L^1(\mathbb{R}^n)$ into $L^{1,\infty }(\mathbb{R}^n)$, when $a\geq 0$, and from $L^p(\mathbb{R}^n)$ into itself, for every $\frac{n}{n-\sigma}<p<\frac{n}{\sigma}$, when $-\frac{(n-2)^2}{4}<a<0$.

\begin{rem}
In the proof of Proposition \ref{Prop2.2} we establish that the operator $T_{*,0}^a$ is not bounded from $L^{n/(n-\sigma)}(\mathbb{R}^n)$ into $L^{n/(n-\sigma),\infty}(\mathbb{R}^n)$.
\end{rem}

We are going to prove the weighted $L^p$-inequalities for $T_{*,\alpha}^a$. As we saw above we have that
$T_{*,\alpha}^a(f)\leq CT_{*,m}^a$, where $m\in \mathbb{N}$ is such that $m-1\leq \alpha <m$.

We use \cite[Theorem 6.6]{BZ}. We consider, as in the proof of Theorem \ref{Th1.1}, for every $t>0$, $A_t^a=I-(I-T_t^a)^k$, where $k\in \mathbb{N}$ and $k>n/2$.

Let $B$ be a ball in $\mathbb{R}^n$. We recall that, for every $j\in \mathbb{N}$, $j\geq 1$, $S_j(B)=2^jB\setminus2^{j-1}B$, and that by $r_B$ we represent the radius of $B$.

Assume that $1<p<q<\infty$, when $a\geq 0$ and that $\frac{n}{n-\sigma}<p<q<\frac{n}{\sigma}$, when $-\frac{(n-2)^2}{4}<a<0$. Let $f$ be a smooth function supported in $B$ and $j\in \mathbb{N}$, $j\geq 2$. By \eqref{2.16} we have that 
$$
\left(\frac{1}{|S_j(B)|}\int_{S_j(B)}|A_{r_B^2}^a(f)(x)|^qdx\right)^{1/q}\leq C\frac{e^{-c2^{2j}}}{2^{jn/q}}\left(\frac{1}{|B|}\int_B|f(x)|^pdx\right)^{1/p}.
$$
On the other hand, by \eqref{1.1} we get
\begin{align*}
    |t^m\partial _t^mT_t^a(f)(x)|&\leq C\int_B\Big(1+\frac{\sqrt{t}}{|x|}\Big)^\sigma \Big(1+\frac{\sqrt{t}}{|y|}\Big)^\sigma \frac{e^{-c\frac{|x-y|^2}{t}}}{t^{n/2}}|f(y)|dy\\
    &\leq C\left\{
    \begin{array}{ll}
    t^{-n/2},&\mbox{ when }a\geq 0\\[0.4cm]
    t^{-n/2+\sigma},&\mbox{ when }\displaystyle -\frac{(n-2)^2}{4}<a<0
    \end{array}\right.,\quad t\geq 1, \;x\in \mathbb{R}^n\setminus\{0\}.
\end{align*}
Here $C>0$ depends on $x\in \mathbb{R}^n\setminus\{0\}$. Then, 
$$
\lim_{t\rightarrow \infty}t^m\partial _t^mT_t^a(f)(x)=0,\quad x\in \mathbb{R}^n\setminus\{0\}.
$$
We can write 
$$
t^m\partial_t^mT_t^a(f)(x)=-\int_t^\infty \partial _u(u^m\partial _u^mT_u^a(f)(x))du,\quad t>0,\;x\in \mathbb{R}^n\setminus\{0\}.
$$
We obtain
$$
T_{*,m}^a(f)(x)\leq \int_0^\infty |\partial _u(u^m\partial _u^mT_u^a(f)(x))|du,\quad x\in \mathbb{R}^n\setminus\{0\}.
$$

The properties established in the proof of Theorem \ref{Th1.1} for the operator $S_{\ell ,a}$, $\ell \in \mathbb{N}$, lead to
$$
\left(\frac{1}{|S_j(B)|}\int_{S_j(B)}|T_{*,m}^a((I-A_{r_B^2}^a)(f))(x)|^qdx\right)^{1/q}\leq C2^{-j(n/q+2k)}\left(\frac{1}{|B|}\int_B|f(x)|^qdx\right)^{1/q}.
$$
The proof can be finished by using \cite[Theorem 6.6]{BZ}.

\section {Appendix}
In this appendix we prove the following continuity result.
\begin{ThA}
For every $f\in L^p(\mathbb{R}^n,w)$ we have that
$$
\lim_{t\rightarrow 0^+}T_t^a(f)(x)=f(x),\quad \mbox{for almost all }x\in \mathbb{R}^n\setminus\{0\},
$$
provided that one the three following conditions holds:

(i) $a\geq 0$ and $p=1$;

(ii) $a\geq 0$, $1<p<\infty$ and $w\in A_p(\mathbb{R}^n)\cap RH_1(\mathbb{R}^n)$;

(iii) $-\frac{(n-2)^2}{4}< a<0$, $\frac{n}{n-\sigma}<p<\frac{n}{\sigma}$ and $w\in A_{(n-\sigma)p/n}(\mathbb{R}^n)\cap RH_{(\frac{n\sigma}{p})'}(\mathbb{R}^n)$.
\end{ThA}
\begin{proof}
According to $L^p$-boundedness properties for the maximal operator $T_{*,0}^a$ established in Theorem \ref{1.3}, by using a standard procedure, in order to prove the result it is sufficient to see that, for every $f\in C_c^\infty (\mathbb{R}^n)$,
$$
\lim_{t\rightarrow 0^+}T_t^a(f)(x)=f(x),\quad x\in \mathbb{R}^n\setminus\{0\}.
$$
Let $f\in C_c^\infty(\mathbb{R}^n)$. Since $\int_{\mathbb{R}^n}W_t(x-y)dy=1$, $x\in \mathbb{R}^n$ and $t>0$, we can write
\begin{align*}
T_t^a(f)(x)-f(x)&=\int_{\mathbb{R}^n}T_t^a(x,y)(f(y)-f(x))dy+f(x)\int_{\mathbb{R}^n}(T_t^a(x,y)-W_t(x-y))dy\\
&=H_1(t,x)+H_2(t,x),\quad x\in \mathbb{R}^n\setminus\{0\}\mbox{ and }t>0.
\end{align*}
Let $x\in \mathbb{R}^n\setminus\{0\}$ and $\varepsilon >0$. According to \eqref{1.1} we have that
$$
|H_1(t,x)|\leq C\int_{\mathbb{R}^n}\Big(1+\frac{\sqrt{t}}{|x|}\Big)^\sigma \Big(1+\frac{\sqrt{t}}{|y|}\Big)^\sigma\frac{e^{-c\frac{|x-y|^2}{t}}}{t^{n/2}}|f(y)-f(x)|dy,\quad t>0.
$$
We choose $0<\delta <|x|/2$ such that $|f(y)-f(x)|<\varepsilon$ when $|x-y|<\delta$. Then,
\begin{align*}
|H_1(t,x)|&\leq C\left(\int_{\mathbb{R}^n\setminus B(x,\delta)}\Big(1+\frac{\sqrt{t}}{|x|}\Big)^\sigma \Big(1+\frac{\sqrt{t}}{|y|}\Big)^\sigma\frac{e^{-c\frac{|x-y|^2}{t}}}{t^{n/2}}dy\right.\\
&\quad \left.+\varepsilon \int_{B(x,\delta)}\Big(1+\frac{\sqrt{t}}{|x|}\Big)^\sigma \Big(1+\frac{\sqrt{t}}{|y|}\Big)^\sigma\frac{e^{-c\frac{|x-y|^2}{t}}}{t^{n/2}}dy\right)\\
&=H_{1,1}(t,x)+H_{1,2}(t,x),\quad t>0.
\end{align*}
Suppose that $a\geq 0$. Then, $\sigma \leq 0$. We get
$$
H_{1,1}(t,x)\leq C\int_{\mathbb{R}^n\setminus B(x,\delta)}\frac{e^{-c\frac{|x-y|^2}{t}}}{t^{n/2}}dy=C\int_{\mathbb{R}^n\setminus B(0,\delta)}\frac{e^{-c\frac{|z|^2}{t}}}{t^{n/2}}dz=C\int_{\mathbb{R}^n\setminus B(0,\frac{\delta}{\sqrt{t}})}e^{-c|u|^2}du.
$$
There exists $t_1>0$ such that $H_{1,1}(t,x)<\varepsilon$, $t\in (0,t_1)$. Hence $|H_1(t,x)|\leq C\varepsilon $, $t\in (0,t_1)$.

Assume now that $-\frac{(n-2)^2}{4}<a<0$. Then $0<\sigma <\frac{n-2}{2}$. We obtain
\begin{align*}
    H_{1,1}(t,x)&\leq C\int_{\mathbb{R}^n\setminus B(x,\delta)}\Big(1+\frac{\sqrt{t}}{|x|}\Big)^\sigma \Big(1+\frac{\sqrt{t}}{|y|}\Big)^\sigma\frac{e^{-c\frac{|x-y|^2}{t}}}{t^{n/2}}dy\\
    &\leq C\left(\int_{[\mathbb{R}^n\setminus B(x,\delta)]\cap B(0,1)}+\int_{[\mathbb{R}^n\setminus B(x,\delta)]\cap B(0,1)^c}\right)\Big(1+\frac{\sqrt{t}}{|y|}\Big)^\sigma\frac{e^{-c\frac{|x-y|^2}{t}}}{t^{n/2}}dy\\
    &\leq C\left(\frac{e^{-\frac{\delta^2}{t}}}{t^{n/2}}(1+t^{\sigma /2})+(1+\sqrt{t})^\sigma \int_{\mathbb{R}^n\setminus B(0,\frac{\delta}{\sqrt{t}})}e^{-c|u|^2}du\right).
\end{align*}
There exists $t_2>0$ such that $H_{1,1}(t,x)<\varepsilon$, $t\in (0,t_2)$.

On the other hand, since $0<\delta<|x|/2$, we can write
$$
H_{1,2}(t,x)\leq C\varepsilon \int_{B(x,\delta)}\Big(1+\frac{\sqrt{t}}{|y|}\Big)^\sigma \frac{e^{-c\frac{|x-y|^2}{t}}}{t^{n/2}}dy\leq C\varepsilon (1+\sqrt{t})^\sigma\int_{\mathbb{R}^n}\frac{e^{-c\frac{|x-y|^2}{t}}}{t^{n/2}}dy\leq C\varepsilon ,\quad t>0.
$$
Hence, $|H_1(t,x)|\leq C\varepsilon$, $t\in (0,t_2)$.

We have that
\begin{align*}
    |H_2(t,x)|&\leq C\int_{\mathbb{R}^n}|T_t^a(x,y)-W_t(x-y)|dy\leq C\sum_{j=1}^3\int_{\mathbb{R}^n}|T_t^a(x,y)-W_t(x-y)|\mathcal{X}_{A_j}(y)dy\\
    &=\sum_{j=1}^3L_j(t,x),\quad t>0,
\end{align*}
where $A_1=\{y\in\mathbb R^n:\;0<|y|<\frac{|x|}{2}\}$, $A_2=\{y\in\mathbb R^n:\;|y|>\frac{3|x|}{2}\}$ and $A_3=\{y\in\mathbb R^n:\;\frac{|x|}{2}\leq |y|\leq \frac{3|x|}{2}\}$.

Assume that $a\geq 0$. It follows that
\begin{equation}\label{A1}
L_1(t,x)\leq C\int_{\mathbb{R}^n}\mathcal{X}_{A_1}(y)\frac{e^{-c\frac{|x-y|^2}{t}}}{t^{n/2}}dy\leq C\frac{e^{-\frac{|x|^2}{t}}}{t^{n/2}},\quad t>0,
\end{equation}
and
\begin{align}\label{A2}
L_2(t,x)&\leq C\int_{\mathbb{R}^n}\mathcal{X}_{A_2}(y)\frac{e^{-c\frac{|x-y|^2}{t}}}{t^{n/2}}dy\leq C\int_{\mathbb{R}^n\setminus B(0,\frac{3|x|}{2})}\frac{e^{-c\frac{|y|^2}{t}}}{t^{n/2}}dy\nonumber\\
&\leq C\int_{\mathbb{R}^n\setminus B(0,\frac{3|x|}{2\sqrt{t}})}e^{-c|u|^2}du,\quad t>0.
\end{align}
On the other hand, by using Duhamel formula we obtain
\begin{align*}
    T_t^a(x,y)-W_t(x-y)&=-a\int_0^{t/2}\int_{\mathbb{R}^n}W_{t-s}(x-z)|z|^{-2}T_s^a(z,y)dzds\\
    &\quad -a\int_0^{t/2}\int_{\mathbb{R}^n}W_s(x-z)|z|^{-2}T_{t-s}^a(z,y)dzds,\quad x,y\in \mathbb{R}^n\mbox{ and }t>0.
\end{align*}
Since $\frac{|x-z|^2}{t-s}+\frac{|z-y|^2}{s}\geq c\frac{|x-y|^2}{t}$, $y,z\in \mathbb{R}^n$, $0<s<t$, by using \eqref{1.1}, \eqref{2.1} and \cite[(28)]{BADLL} we can deduce that
\begin{align*}
|T_t^a(x,y)-W_t(x-y)|&\leq C\frac{e^{-c\frac{|x-y|^2}{t}}}{t^{n/2}}\int_0^{t/2}\int_{\mathbb{R}^n}\Big(\frac{e^{-c(\frac{|x-z|^2}{s}+\frac{|z-y|^2}{s})}}{s^{n/2}}\Big)|z|^{-2}dzds\\
&\leq C\frac{e^{-c\frac{|x-y|^2}{t}}}{t^{n/2-1}}\Big(\frac{1}{|x|^2}+\frac{1}{|y|^2}\Big),\quad y\in \mathbb{R}^n\setminus\{0\}\mbox{ and }t>0.
\end{align*}
We get
\begin{align}\label{A3}
  L_3(t,x)&\leq C\int_{\mathbb{R}^n}\mathcal{X}_{A_3}(y)\frac{e^{-c\frac{|x-y|^2}{t}}}{|x|^2t^{{n/2}-1}}dy\leq C\int_{\mathbb{R}^n}\mathcal{X}_{A_3}(y)\frac{t^{3/4}}{|x-y|^{n-1/2}}dy\nonumber\\
  &\leq Ct^{3/4}\int_{B(0,\frac{5|x|}{2})}\frac{dy}{|x-y|^{n-1/2}} \leq Ct^{3/4}\int_0^{5|x|/2}\rho ^{-1/2}d\rho\leq Ct^{3/4},\quad t>0.
\end{align}

By combining \eqref{A1}, \eqref{A2} and \eqref{A3}, there exists $t_3>0$ such that $|H_2(t,x)|<\varepsilon$, $t\in (0,t_3)$.

When $-\frac{(n-2)^2}{4}<a<0$ by proceeding in a similar way we can see that there exists $t_4>0$ such that $|H_2(t,x)|<\varepsilon$, $t\in (0,t_4)$. 

By putting together the above estimations we conclude that
$$
\lim_{t\rightarrow 0^+}T_t^a(f)(x)=f(x).
$$

\end{proof} 
\bibliographystyle{acm}

\begin{thebibliography}{10}

\bibitem{AJS}
{\sc Akcoglu, M.~A., Jones, R.~L., and Schwartz, P.~O.}
\newblock Variation in probability, ergodic theory and analysis.
\newblock {\em Illinois J. Math. 42}, 1 (1998), 154--177.

\bibitem{AB}
{\sc Almeida, V., and Betancor, J.~J.}
\newblock Variation and oscillation for harmonic operators in the inverse
  gaussian setting.
\newblock Preprint 2020 \href{http://arxiv.org/abs/}{(arXiv:2012.11205)}.

\bibitem{BZ}
{\sc Bernicot, F., and Zhao, J.}
\newblock New abstract {H}ardy spaces.
\newblock {\em J. Funct. Anal. 255}, 7 (2008), 1761--1796.

\bibitem{Bo}
{\sc Bourgain, J.}
\newblock Pointwise ergodic theorems for arithmetic sets.
\newblock {\em Inst. Hautes \'{E}tudes Sci. Publ. Math.}, 69 (1989), 5--45.
\newblock With an appendix by the author, Harry Furstenberg, Yitzhak Katznelson
  and Donald S. Ornstein.

\bibitem{BBD}
{\sc Bui, H.-Q., Bui, T.~A., and Duong, X.~T.}
\newblock Weighted {B}esov and {T}riebel-{L}izorkin spaces associated with
  operators and applications.
\newblock {\em Forum Math. Sigma 8\/} (2020), Paper No. e11, 95.

\bibitem{Bui1}
{\sc Bui, T.~A.}
\newblock Besov and {T}riebel-{L}izorkin spaces for {S}chr\"{o}dinger operators
  with inverse-square potentials and applications.
\newblock {\em J. Differential Equations 269}, 1 (2020), 641--688.

\bibitem{BADLL}
{\sc Bui, T.~A., D'Ancona, P., Duong, X.~T., Li, J., and Ly, F.~K.}
\newblock Weighted estimates for powers and smoothing estimates of
  {S}chr\"{o}dinger operators with inverse-square potentials.
\newblock {\em J. Differential Equations 262}, 3 (2017), 2771--2807.

\bibitem{CJRW1}
{\sc Campbell, J.~T., Jones, R.~L., Reinhold, K., and Wierdl, M.}
\newblock Oscillation and variation for the {H}ilbert transform.
\newblock {\em Duke Math. J. 105}, 1 (2000), 59--83.

\bibitem{CJRW2}
{\sc Campbell, J.~T., Jones, R.~L., Reinhold, K., and Wierdl, M.}
\newblock Oscillation and variation for singular integrals in higher
  dimensions.
\newblock {\em Trans. Amer. Math. Soc. 355}, 5 (2003), 2115--2137.

\bibitem{CMMTV}
{\sc Crescimbeni, R., Mac\'{\i}as, R.~A., Men\'{a}rguez, T., Torrea, J.~L., and
  Viviani, B.}
\newblock The {$\rho$}-variation as an operator between maximal operators and
  singular integrals.
\newblock {\em J. Evol. Equ. 9}, 1 (2009), 81--102.

\bibitem{DHK}
{\sc Dr\'{a}bek, P., Heinig, H.~P., and Kufner, A.}
\newblock Higher-dimensional {H}ardy inequality.
\newblock In {\em General inequalities, 7 ({O}berwolfach, 1995)}, vol.~123 of
  {\em Internat. Ser. Numer. Math.} Birkh\"{a}user, Basel, 1997, pp.~3--16.

\bibitem{Duo}
{\sc Duoandikoetxea, J.}
\newblock {\em Fourier analysis}, vol.~29 of {\em Graduate Studies in
  Mathematics}.
\newblock American Mathematical Society, Providence, RI, 2001.
\newblock Translated and revised from the 1995 Spanish original by David
  Cruz-Uribe.

\bibitem{FGR}
{\sc Fornaro, S., Gregorio, F., and Rhandi, A.}
\newblock Elliptic operators with unbounded diffusion coefficients perturbed by
  inverse square potentials in {$L^p$}-spaces.
\newblock {\em Commun. Pure Appl. Anal. 15}, 6 (2016), 2357--2372.

\bibitem{HLMMY}
{\sc Hofmann, S., Lu, G., Mitrea, D., Mitrea, M., and Yan, L.}
\newblock Hardy spaces associated to non-negative self-adjoint operators
  satisfying {D}avies-{G}affney estimates.
\newblock {\em Mem. Amer. Math. Soc. 214}, 1007 (2011), vi+78.

\bibitem{JN}
{\sc Johnson, R., and Neugebauer, C.~J.}
\newblock Change of variable results for {$A_p$}- and reverse {H}\"{o}lder
  {${\rm RH}_r$}-classes.
\newblock {\em Trans. Amer. Math. Soc. 328}, 2 (1991), 639--666.

\bibitem{JKRW}
{\sc Jones, R.~L., Kaufman, R., Rosenblatt, J.~M., and Wierdl, M.}
\newblock Oscillation in ergodic theory.
\newblock {\em Ergodic Theory Dynam. Systems 18}, 4 (1998), 889--935.

\bibitem{JR}
{\sc Jones, R.~L., and Reinhold, K.}
\newblock Oscillation and variation inequalities for convolution powers.
\newblock {\em Ergodic Theory Dynam. Systems 21}, 6 (2001), 1809--1829.

\bibitem{JSW}
{\sc Jones, R.~L., Seeger, A., and Wright, J.}
\newblock Strong variational and jump inequalities in harmonic analysis.
\newblock {\em Trans. Amer. Math. Soc. 360}, 12 (2008), 6711--6742.

\bibitem{JW}
{\sc Jones, R.~L., and Wang, G.}
\newblock Variation inequalities for the {F}ej\'{e}r and {P}oisson kernels.
\newblock {\em Trans. Amer. Math. Soc. 356}, 11 (2004), 4493--4518.

\bibitem{KMVZZ}
{\sc Killip, R., Miao, C., Visan, M., Zhang, J., and Zheng, J.}
\newblock Sobolev spaces adapted to the {S}chr\"{o}dinger operator with
  inverse-square potential.
\newblock {\em Math. Z. 288}, 3-4 (2018), 1273--1298.

\bibitem{LeMX2}
{\sc Le~Merdy, C., and Xu, Q.}
\newblock Maximal theorems and square functions for analytic operators on
  {$L^p$}-spaces.
\newblock {\em J. Lond. Math. Soc. (2) 86}, 2 (2012), 343--365.

\bibitem{LeMX}
{\sc Le~Merdy, C., and Xu, Q.}
\newblock Strong {$q$}-variation inequalities for analytic semigroups.
\newblock {\em Ann. Inst. Fourier (Grenoble) 62}, 6 (2012), 2069--2097 (2013).

\bibitem{Lep}
{\sc L\'{e}pingle, D.}
\newblock La variation d'ordre {$p$} des semi-martingales.
\newblock {\em Z. Wahrscheinlichkeitstheorie und Verw. Gebiete 36}, 4 (1976),
  295--316.

\bibitem{LS}
{\sc Liskevich, V., and Sobol, Z.}
\newblock Estimates of integral kernels for semigroups associated with
  second-order elliptic operators with singular coefficients.
\newblock {\em Potential Anal. 18}, 4 (2003), 359--390.

\bibitem{LSV}
{\sc Liskevich, V., Sobol, Z., and Vogt, H.}
\newblock On the {$L_p$}-theory of {$C_0$}-semigroups associated with
  second-order elliptic operators. {II}.
\newblock {\em J. Funct. Anal. 193}, 1 (2002), 55--76.

\bibitem{MTX}
{\sc Ma, T., Torrea, J.~L., and Xu, Q.}
\newblock Weighted variation inequalities for differential operators and
  singular integrals in higher dimensions.
\newblock {\em Sci. China Math. 60}, 8 (2017), 1419--1442.

\bibitem{MST}
{\sc Mac\'{\i}as, R., Segovia, C., and Torrea, J.~L.}
\newblock Heat-diffusion maximal operators for {L}aguerre semigroups with
  negative parameters.
\newblock {\em J. Funct. Anal. 229}, 2 (2005), 300--316.

\bibitem{MZZ}
{\sc Miao, C., Zhang, J., and Zheng, J.}
\newblock Maximal estimates for {S}chr\"{o}dinger equations with inverse-square
  potential.
\newblock {\em Pacific J. Math. 273}, 1 (2015), 1--19.

\bibitem{MS}
{\sc Milman, P.~D., and Semenov, Y.~A.}
\newblock Global heat kernel bounds via desingularizing weights.
\newblock {\em J. Funct. Anal. 212}, 2 (2004), 373--398.

\bibitem{NSj}
{\sc Nowak, A., and Sj\"{o}gren, P.}
\newblock The multi-dimensional pencil phenomenon for {L}aguerre heat-diffusion
  maximal operators.
\newblock {\em Math. Ann. 344}, 1 (2009), 213--248.

\bibitem{OSTTW}
{\sc Oberlin, R., Seeger, A., Tao, T., Thiele, C., and Wright, J.}
\newblock A variation norm {C}arleson theorem.
\newblock {\em J. Eur. Math. Soc. (JEMS) 14}, 2 (2012), 421--464.

\bibitem{OK1}
{\sc Okazawa, N.}
\newblock On the perturbation of linear operators in {B}anach and {H}ilbert
  spaces.
\newblock {\em J. Math. Soc. Japan 34}, 4 (1982), 677--701.

\bibitem{OK2}
{\sc Okazawa, N.}
\newblock {$L^p$}-theory of {S}chr\"{o}dinger operators with strongly singular
  potentials.
\newblock {\em Japan. J. Math. (N.S.) 22}, 2 (1996), 199--239.

\bibitem{Pil}
{\sc Pilarczyk, D.}
\newblock Self-similar asymptotics of solutions to heat equation with inverse
  square potential.
\newblock {\em J. Evol. Equ. 13}, 1 (2013), 69--87.

\bibitem{Qi}
{\sc Qian, J.}
\newblock The {$p$}-variation of partial sum processes and the empirical
  process.
\newblock {\em Ann. Probab. 26}, 3 (1998), 1370--1383.

\bibitem{RS}
{\sc Reed, M., and Simon, B.}
\newblock {\em Methods of modern mathematical physics. {II}. {F}ourier
  analysis, self-adjointness}.
\newblock Academic Press [Harcourt Brace Jovanovich, Publishers], New
  York-London, 1975.

\bibitem{Si}
{\sc Simon, B.}
\newblock Essential self-adjointness of {S}chr\"{o}dinger operators with
  singular potentials.
\newblock {\em Arch. Rational Mech. Anal. 52\/} (1973), 44--48.

\bibitem{SV}
{\sc Sobol, Z., and Vogt, H.}
\newblock On the {$L_p$}-theory of {$C_0$}-semigroups associated with
  second-order elliptic operators. {I}.
\newblock {\em J. Funct. Anal. 193}, 1 (2002), 24--54.

\bibitem{SY1}
{\sc Song, L., and Yan, L.}
\newblock A maximal function characterization for {H}ardy spaces associated to
  nonnegative self-adjoint operators satisfying {G}aussian estimates.
\newblock {\em Adv. Math. 287\/} (2016), 463--484.

\bibitem{SY2}
{\sc Song, L., and Yan, L.}
\newblock Maximal function characterizations for {H}ardy spaces associated with
  nonnegative self-adjoint operators on spaces of homogeneous type.
\newblock {\em J. Evol. Equ. 18}, 1 (2018), 221--243.

\bibitem{StLP}
{\sc Stein, E.~M.}
\newblock {\em Topics in harmonic analysis related to the {L}ittlewood-{P}aley
  theory}.
\newblock Annals of Mathematics Studies, No. 63. Princeton University Press,
  Princeton, N.J.; University of Tokyo Press, Tokyo, 1970.

\bibitem{TZ}
{\sc Torrea, J.~L., and Zhang, C.}
\newblock Fractional vector-valued {L}ittlewood-{P}aley-{S}tein theory for
  semigroups.
\newblock {\em Proc. Roy. Soc. Edinburgh Sect. A 144}, 3 (2014), 637--667.

\bibitem{V}
{\sc Vogt, H.}
\newblock {\em $L^p$-properties of second order elliptic differential
  operators.}
\newblock PhD thesis, Fakultat Mathematik und Naturwissenschaften der
  Technischen Universit\"at Dresden, 2001.

\bibitem{ZK}
{\sc Zorin-Kranich, P.}
\newblock Variational and jump inequalities.
\newblock 2019
  \href{https://www.math.uni-bonn.de/~pzorin/slides/2019-05-10_HIM_Bonn_jumps.pdf}{(www.math.uni-bonn.de/~pzorin/slides)}.

\end{thebibliography}

\end{document}